\documentclass[reqno]{amsart}
\usepackage{amsmath}

\usepackage{amsfonts}
\usepackage{amssymb}
\usepackage{hyperref}
\usepackage{xcolor}
\usepackage{graphicx}
\usepackage{mathptmx}      
\usepackage{latexsym}
\usepackage{cite}
\usepackage{epstopdf}
\usepackage[a4paper,top=3.5cm,bottom=3cm,left=3cm,right=3cm]{geometry}

\begin{document}
\title[Existence and stability results for $\psi$-Hilfer fractional differential equations]{Existence and stability results for a sequential
$\psi$-Hilfer fractional integro-differential equations with nonlocal boundary conditions}

\author[F. Haddouchi]{Faouzi Haddouchi}

\address{
Department of Physics, University of Sciences and Technology of
Oran-MB, Oran, Algeria
\newline
And
\newline
Laboratory of Fundamental and Applied Mathematics of Oran (LMFAO), University of Oran 1, Oran, Algeria}
\email{fhaddouchi@gmail.com}

\author[M.E.Samei]{Mohammad Esmael Samei}

\address{Department of Mathematics, Faculty of Basic Science, Bu-Ali Sina University, Hamedan, Iran}
\email{mesamei@basu.ac.ir; mesamei@gmail.com}

\author[Sh. Rezapour]{Shahram Rezapour}

\address{
Department of Mathematics, Azarbaijan Shahid Madani University, Tabriz, Iran
\newline
And
\newline
Department of Medical Research, China Medical University Hospital, China Medical University, Taichung, Taiwan}
\email{rezapourshahram@yahoo.ca; sh.rezapour@azaruniv.edu}


\begin{abstract}
This paper deals with the existence and uniqueness of solutions for a nonlinear boundary value problem involving a sequential $\psi$-Hilfer fractional integro-differential equations with nonlocal boundary conditions. The existence and uniqueness of solutions are established for the considered problem by using the Banach contraction principle, Sadovski's fixed point theorem, and Krasnoselskii-Schaefer fixed point theorem due to Burton and Kirk. In addition, the Ulam-Hyers stability of solutions is discussed. Finally, the obtained results are illustrated by examples.
\end{abstract}

\maketitle \numberwithin{equation}{section}
\newtheorem{theorem}{Theorem}[section]
\newtheorem{lemma}[theorem]{Lemma}
\newtheorem{definition}[theorem]{Definition}
\newtheorem{proposition}[theorem]{Proposition}
\newtheorem{corollary}[theorem]{Corollary}
\newtheorem{remark}[theorem]{Remark}
\newtheorem{exmp}{Example}[section]

\section{Introduction\label{sec:1}}

The study of differential equations with fractional order has become very useful over the last few decades due to its extensive applications in various areas such as mathematics, physics, biology, networks, image processing, control theory, viscoelasticity, quantum physics, and so forth, see \cite{Lazopoulos1,Lazopoulos2,Hilf,Hilf3,Kilbas,Pod,Sam,Stempin,Tarasov}.
Recently, fractional and integral differential operators have been proposed to describe new physical phenomena \cite{Rosa,Silva,Sousa2,Hermann, Atanack,Sousa3}. There are several approaches to fractional derivatives and integral operators such as Caputo, Caputo-Fabrizio, Riemann-Liouville, Weyl, Atangana-Baleanu, Katugampola, Hadamard, etc. For some recent works on this class of fractional differential equations, we refer the reader to the articles \cite{Almeida,Abd,Atang1,Atang2,Atang3,Atang4,Atang5,Hilf2,Hilf4,Hadd1,Hadd2,Hadd3,Hadd4,Hadd5,Ismail,Kumar,Katug,Kiata}

In \cite{Sousa3}, Sousa and Oliveira have recently proposed a fractional differentiation operator, called $\psi$-Hilfer operator that has the special property of unifying  and generalizing several different fractional operators.

By using fixed point theory, the existence and stability results of solutions involving fractional differential equations with a variety of boundary conditions have been studied in recent years, see \cite{Vivek,Nieto,Yang,Asawa,Furati,Saen,Sousa4,Harik,Reza,Abdo2,Sousa5} and references cited therein.

Thaiprayoon et.al., in \cite{thai} studied the existence and stability of solutions
of a boundary value problem (BVP) of the fractional thermostat control model with
$\psi$-Hilfer fractional operator of the form

\begin{equation}\label{eq01}
       \begin{cases} ^{H}D^{\alpha,\rho;\psi}_{0^{+}}x(\varsigma)= f(\varsigma,x(\theta \varsigma), I_{0^{+}}^{q,\psi}x(\varepsilon \varsigma)) \   \in (0,T],\\
   \sum_{i=1}^{m}\omega_{i}^{H}D^{\beta_{i},\rho;\psi}_{0^{+}}x(\xi_{i})=A,\ \sum_{j=1}^{n}\gamma_{j}^{H}D^{\mu_{j},\rho;\psi}_{0^{+}}x(\sigma_{j})+ \sum_{k=1}^{r}\delta_{k}x(\eta_{k})=B,
       \end{cases}
       \end{equation}

where $^{H}D^{\nu,\rho;\psi}_{0^{+}}$ denotes the $\psi$-Hilfer derivative operators of order $\nu = {\beta,\beta_{i},\mu_{j}}$, $\alpha\in (1,2]$, $\beta_{i},\mu_{j}\in (0,1]$, $A,B,\omega_{i},\lambda_{j},\delta_{k}\in \mathbb{R}$, $\xi_{i}, \sigma_{j},\eta_{k}\in (0,T)$, $i=1,2,...,m$, $j=1,2,...,n$, $k=1,2,...,r$, $\rho \in[0,1]$, $I_{0^{+}}^{q,\psi}$ is the $\psi$-RL-integral of order $q > 0$, $\theta,\varepsilon \in(0, 1]$, $f \in C(J \times \mathbb{R}^{2},\mathbb{R})$, and $J :=[0,T]$ with $T > 0$. They established existence results by employing some fixed point theorems. Also, various types of stability in the format of Ulam for this problem are discussed. \\

In \cite{Kot}, the authors established sufficient conditions to approve the existence and uniqueness of
solutions of a nonlinear implicit $\psi$-Hilfer fractional boundary value problem of the cantilever beam
model with nonlinear boundary conditions

\begin{equation}\label{eq02}
       \begin{cases} ^{H}D^{\alpha,\rho;\psi}_{a^{+}}x(t)= f(t,x(t),^{H}D^{\alpha,\rho;\psi}_{a^{+}}x(t), I_{a^{+}}^{\beta,\psi}x(t)), \ \  t\in (a,b),\\
   x(a)=0,\ ^{H}D^{\delta,\rho;\psi}_{a^{+}}x(a)=0,\\ \sum_{i=1}^{n}\mu_{i} ^{H}D^{\theta_{i},\rho;\psi}_{a^{+}}x(\kappa_{i})=H(\eta,x(\eta)),\
   \sum_{j=1}^{n} \varphi_{j}^{H}D^{\phi_{j},\rho;\psi}_{a^{+}}x(\varsigma_{j})=G(\xi,x(\xi)),
       \end{cases}
       \end{equation}
where  $^{H}D^{\nu,\rho;\psi}_{a^{+}}$ denotes the $\psi$-Hilfer fractional derivative operators of order $\nu = \{\alpha,\delta,\theta_{i},\phi_{j}\}$, $\alpha\in(3,4]$, $\theta_{i}\in(0,1]$, $\delta\in(1,2]$, $\phi_{j}\in(2,3]$, $\kappa_{i},\varsigma_{j},\eta,\xi \in(a,b]$, $\mu_{i},\varphi_{j}\in \mathbb{R}$, for $i=1,2,...,m$, $j=1,2,...,n$, and $\rho\in [0,1]$. $I_{a^{+}}^{\beta,\psi}$ denotes $\psi$-Riemann–Liouville fractional integral of order $\beta>0$, $f\in C(J\times\mathbb{R}^{3},\mathbb{R})$, $H,G\in C(J,\mathbb{R})$ and $J:=[a,b]$, $b>a>0$.

By using Banach's fixed point theorem, the uniqueness result is proved. Also, they obtained the existence result is obtained by applying the fixed point theorem of Schaefer. Furthermore, the different types of Ulam's stability are used to investigate the stability of the solution of the proposed problem.\\

In \cite{Ntouy}, Ntouyas and Vivek studied the existence and uniqueness of solutions for a new class of boundary value problems of sequential $\psi$-Hilfer-type fractional differential equations with
multi-point boundary conditions of the form

\begin{equation}\label{eq03}
       \begin{cases} \big(\ ^{H}D^{\alpha,\beta;\psi}_{a^{+}}+k\ ^{H}D^{\alpha-1,\beta,\psi}_{a^{+}}\big)x(t)= f(t,x(t)), \ \  t\in [a,b],\\
   x(a)=0,\ x(b)=\sum_{i=1}^{m}\lambda_{i} x(\theta_{i}),
       \end{cases}
       \end{equation}
where $ ^{H}D^{\alpha,\beta;\psi}_{a^{+}}$ is the  $\psi$-Hilfer fractional derivative of order $\alpha$,
$1 < \alpha < 2$, and parameter $\beta$, $0\leq \beta \leq1$, $f:[a,b]\times\mathbb{R}\rightarrow \mathbb{R}$ is a continuous function, $a < b$, $k, \lambda_{i}\in \mathbb{R}$, $i = 1,2,...,m$, and $a < \theta_{1} < \theta_{2} < ... < \theta_{m} < b$.

Existence and uniqueness results are obtained by using the classical fixed point theorems of Banach, Krasnoselskii, and the nonlinear alternative of Leray-Schauder.\\

In \cite{Sousa}, the authors considered the following nonlinear fractional differential equations in the $\psi$-Hilfer sense
\begin{equation}\label{eq04}
       \begin{cases}  ^{H}D^{\alpha,\beta;\psi}_{a^{+}}y(t)= f(t,y(t),\ ^{H}D^{\alpha,\beta;\psi}_{a^{+}}y(t)), \ \  t\in J,\\
   I_{a^{+}}^{1-\gamma,\psi}y(a)=y_{a},
       \end{cases}
       \end{equation}

where $^{H}D^{\alpha,\beta;\psi}_{a^{+}}(.)$ is the $\psi$-Hilfer fractional derivative of order $0 < \alpha\leq1$ and type $0\leq\beta\leq 1$, $I_{a^{+}}^{1-\gamma,\psi}(.)$ is the Riemann–Liouville fractional integral of order $1-\gamma$, $\gamma=\alpha+\beta (1-\alpha)$, with respect to function $\psi$, $f:J\times \mathbb{R}\times \mathbb{R}\rightarrow \mathbb{R}$ is a given function space, $J = [a, T]$ with $T > a$ and $y_{a}\in\mathbb{R}$.\\
They studied the existence and uniqueness of solutions for the nonlinear Cauchy problem, \eqref{eq04}, by means of Banach's contraction principle. In addition, the Ulam–Hyers and Ulam–Hyers–Rassias stabilities
of solutions are discussed.\\

In \cite{Lima}  et.al., investigated sufficient conditions for the existence, uniqueness and Ulam–Hyers stability of the solution of a fractional delay impulsive differential equation
\begin{equation}\label{eq05}
\begin{cases}^{H}D^{\alpha,\beta;\psi}_{0^{+}}x(t)= F(t,x_{t}),\ t\in (0,T]\backslash \{t_{1},t_{2},...,t_{m}\},\\
\Delta x(t_{k})=x(t_{k}^{+})-x(t_{k}^{-})=I_{k}(x(t_{k}^{-})),\ k=1,2,...,m\\
I_{0^{+}}^{1-\gamma}x(0)=y_{0},\\
x(t)=h(t),\ \ \ \ t\in[-r,0],
\end{cases}
 \end{equation}

where $r > 0$, $T > 0$, $F : [0, T]\times \Omega\rightarrow \mathbb{R}$ a given function, $I_{k} : \mathbb{R}\rightarrow \mathbb{R}$ and $h$ is a continuous function defined on $[-r,0]$, $x(t_{k}^{+})=\lim_{\tau\rightarrow 0^{+}} x(t_{k}+\tau)$, $x(t_{k}^{-})=\lim_{\tau\rightarrow 0^{-}} x(t_{k}-\tau)$, $t_{k}$ satisfies $0=t_{0}<t_{1}<t_{2}<...<t_{m}<t_{m+1}=T<\infty$ and $x_{t}(s)=x(t+s)$, $s\in[-r,0]$. Also, $^{H}D^{\alpha,\beta;\psi}_{0^{+}}(.)$ is the $\psi$-Hilfer fractional derivative of order $\alpha \in(0, 1)$ and type $\beta \in [0, 1]$, while $I_{0^{+}}^{1-\gamma}(.)$ is the Riemann–Liouville fractional integral of order $1-\gamma$, where $\gamma=\alpha+\beta(1-\alpha)$.

The proof of their main results are based upon Banach fixed point principle and a Gronwall inequality involving the $\psi$-Riemann–Liouville fractional integral.\\

Inspired by the aforementioned works, in this paper, we are concerned with a new class of boundary value problems of sequential $\psi$-Hilfer-type fractional integro-differential equations with nonlocal boundary conditions of the form
 {\color{blue}\begin{equation}\label{eq09}
  \big( ^{H}D^{\nu,\beta;\psi}_{a^{+}}+\lambda ^{H}D^{\nu-1,\beta,\psi}_{a^{+}}\big)u(t)= f(t,u(t),(\mathcal{V}u)(t), I_{a^{+}}^{2-\mu,\psi}u(t)) ,\  t \in J,
  \end{equation}
  \begin{equation}\label{eq010}
  u(a)=0,\  I^{2-\mu,\psi}_{a^{+}}u(T) =\sum_{i=1}^{m}\alpha_{i}u(\eta_{i})+\sum_{i=1}^{m}\beta_{i}u^{\prime}(\eta_{i})+g(u(\xi)),
  \end{equation}}
 where $1<\nu\leq2$, $0\leq\beta<1$, $J=[a,T],\ 0\leq a<T<\infty$, $f \in\mathcal{C}( J\times \mathbb{R}^{3}, \mathbb{R})$ is a continuous function fulfilling some assumptions that will be described later, $a<\eta_{1}<\eta_{2}< \ldots<\eta_{m}<\xi< T$, $\lambda$, $\alpha_{i}, \beta_{i}\in \mathbb{R}\ (i=1,...,m)$ are given constants, $g:\mathbb{R} \rightarrow \mathbb{R}$ is continuous function with $g(a)=0$.
$^{H}D^{\nu,\beta,\psi}$ is the $\psi$-Hilfer fractional derivative of order $\nu$, and $I_{a^{+}}^{2-\mu,\psi}$ denotes the Riemann-Liouville fractional integral of order $2-\mu$ ($\mu=\nu+\beta(2-\nu)$). $\mathcal{V}:\mathcal{C}( J, \mathbb{R})\rightarrow \mathcal{C}( J, \mathbb{R})$ is an operator (not necessarily linear) satisfying some assumptions that will be specified later.
\\

This paper is organized as follows. In section 2, we present essential notions and lemmas related to the $\psi$-Hilfer fractional operator and some fixed pont theorems and lemmas that will be used to prove our main results. In section 3, we investigate the existence and uniqueness of solutions for \eqref{eq09}-\eqref{eq010} via Banach fixed point principle and fixed point theorems due to Sadovski, a Krasnoselskii-Schaefer fixed point theorem due to Burton and Kirk. Also, sufficient conditions for the Ulam-Hyers stability to the proposed problem is proposed in Sect.3. Finally, we provide some examples illustrating the proposed theoretical results.

 \section{Preliminaries}
In this section, we present some definitions, lemmas and concepts of $\psi$-Hilfer fractional calculus that will be useful in this paper.\\
Let $\mathcal{C}(J,\mathbb{R})$ be the Banach space of continuous functions on $J$ equipped with the norm $\|u\|=\sup_{t\in J}|u(t)|$.\\
Let $\mathcal{A}\mathcal{C}^{n}(J,\mathbb{R})$ denotes the space of $n$-times absolutely continuous functions.
\begin{definition}\label{def1.1}\cite{Kilbas}
Let $(a,b)$ be a finite or infinite interval of the half-axis $\mathbb{R}^{+}$. Also let  $\psi(t)$ be an increasing and positive monotone function on $(a, b]$, having a continuous derivative $\psi^{\prime}(t)$ on $(a, b)$. The $\psi$-Riemann-Liouville fractional integral of order $\alpha$ of a function $f$ depending on the function $\psi$ on $J$ is defined by
\begin{equation}\label{eqn01}
I_{a^{+}}^{\alpha,\psi}f(t)=\frac{1}{\Gamma(\alpha)}\int_{a}^{t}\psi^{\prime}(s)
\big(\psi(t)-\psi(s)\big)^{\alpha-1}f(s)ds,\ t>a>0,\ \alpha>0,
\end{equation}
where $\Gamma(.)$ is the (Euler) Gamma function.
\end{definition}

\begin{definition}\label{def1.2}\cite{Kilbas}
Let $\psi(t)$ be define as in Definition \ref{def1.1} with $\psi^{\prime}(t)\neq0$. The $\nu^{th}$-$\psi$-Riemann–Liouville fractional derivative of a
function $f$ depending on the function $\psi$ is defined as
\begin{eqnarray*}
D_{a^{+}}^{\nu,\psi}f(t)&=&\bigg(\frac{1}{\psi^{\prime}(t)}\frac{d}{dt} \bigg)^{n}I_{a^{+}}^{n-\nu,\psi}f(t)\\
&=&\frac{1}{\Gamma(n-\nu)}\bigg(\frac{1}{\psi^{\prime}(t)}\frac{d}{dt} \bigg)^{n}\int_{a}^{t}\psi^{\prime}(s)
\big(\psi(t)-\psi(s)\big)^{n-\nu-1}f(s)ds,\ \nu>0,
\end{eqnarray*}
where $n = [\nu] + 1$, $[\nu]$ represents the integer part of the real number $\nu$.
\end{definition}

\begin{definition}\label{def1.3}\cite{Vanterler}
Let $\gamma=\nu+\beta(n-\nu)$, $\alpha\in(n-1,n)$ with $n\in \mathbb{N}$, $f\in \mathcal{C}^{n}(J,\mathbb{R})$ and $\psi\in \mathcal{C}^{1}(J,\mathbb{R})$ be increasing with $\psi^{\prime}(t)\neq0$ for each $t\in J$. Then, the $\nu^{th}$-$\psi$-Hilfer fractional derivative of type $\beta\in[0, 1]$ of a function $f$ , depending on the function $\psi$, is defined as
$$^{H}D_{a^{+}}^{\nu,\beta;\psi}f(t)=I_{a^{+}}^{\beta(n-\nu),\psi}\bigg(\frac{1}{\psi^{\prime}(t)}\frac{d}{dt} \bigg)^{n}I_{a^{+}}^{(1-\beta)(n-\nu),\psi}f(t)=I_{a^{+}}^{(\gamma-\nu),\psi}D_{a^{+}}^{\gamma,\psi}f(t)$$
\end{definition}

\begin{lemma} \label{lem1.1}\cite{Kilbas}
Let $\alpha,\beta> 0$. Then, we have the following semigroup property given by
\[I_{a^{+}}^{\alpha,\psi}I_{a^{+}}^{\beta,\psi}f(t)=I_{a^{+}}^{\alpha+\beta,\psi}f(t),\ t>a.\]
\end{lemma}

\begin{proposition} \label{prop1.1}\cite{Kilbas,Vanterler}
Let $\alpha,\upsilon>0$, $t>a$ and consider the function $f(t)=\big( \psi(t)-\psi(a)\big)^{\upsilon-1}$. Then we have the following properties:
\begin{itemize}
  \item [(i)] $I_{a^{+}}^{\alpha,\psi}f(t)=\frac{\Gamma(\upsilon)}{\Gamma(\upsilon+\alpha)}\big( \psi(t)-\psi(a)\big)^{\upsilon+\alpha-1}$,
  \item [(ii)] $D_{a^{+}}^{\alpha,\psi}f(t)=\frac{\Gamma(\upsilon)}{\Gamma(\upsilon-\alpha)}\big( \psi(t)-\psi(a)\big)^{\upsilon-\alpha-1}$,
  \item [(iii)] $^{H}D_{a^{+}}^{\alpha,\beta;\psi}f(t)=\frac{\Gamma(\upsilon)}{\Gamma(\upsilon-\alpha)}\big( \psi(t)-\psi(a)\big)^{\upsilon-\alpha-1}$, $\upsilon>n,\ \alpha \in(n-1,n), 0\leq\beta\leq1$,
    \item [(iv)] In particular, if $\gamma=\alpha+\beta(n-\alpha)$ and $f(t)=\big( \psi(t)-\psi(a)\big)^{\upsilon-k},\ \text{with}\ k=1,...,n,\ (\upsilon>\alpha)$, we have $^{H}D_{a^{+}}^{\alpha,\beta;\psi}\big( \psi(t)-\psi(a)\big)^{\gamma-k}=0.$
\end{itemize}
\end{proposition}

\begin{lemma} \label{lem1.2}\cite{Vanterler}
Let $f\in {\mathcal{C}}^{n}(J,\mathbb{R})$, $\alpha\in(n-1,n)$, $\beta\in[0, 1]$, and $\gamma=\alpha+\beta(n-\alpha)$. Then
$$I_{a^{+}}^{\alpha,\psi}{^{H}D_{a^{+}}^{\alpha,\beta;\psi}}f(t)=f(t)-\sum_{k=1}^{n}\frac{\big( \psi(t)-\psi(a)\big)^{\gamma-k}}{\Gamma(\gamma-k+1)}f_{\psi}^{[n-k]}I_{a^{+}}^{(1-\beta)(n-\alpha),\psi}f(a),$$
for all $t\in J$, where $f_{\psi}^{[n]}f(t)=\bigg(\frac{1}{\psi^{\prime}(t)}\frac{d}{dt} \bigg)^{n}f(t)$.
\end{lemma}

\begin{theorem} \label{thm1.1}\cite{Smart} Let $\Omega\neq\emptyset$ be a closed subset of a Banach
space $X$. If $A:\Omega\rightarrow \Omega $ is a contraction mapping. Then, $A$ admits a unique fixed point.
\end{theorem}

\begin{definition}\label{def1.4} \cite{Kuratowski,Banas}
Let $(X,d)$ be a metric space and $M$ a bounded subset of $X$. Then the Kuratowski measure of noncompactness of $M$, denoted by $\alpha(M)$, is the infimum  of the set of all numbers $\epsilon>0$ such that $M$ can be covered by a finite number of sets with diameters $<0$, that is
\[\alpha(M)=inf\bigg\{\epsilon>0: M\subset \bigcup_{i=1}^{n}M_{i}, M_{i}\subset X, diam(M_{i})<\epsilon (i=1,\cdots,n;n\in\mathbb{N}) \bigg \}.\]
The function $\alpha$ is called Kuratowski's measure of noncompactness.
\end{definition}

\begin{definition}\label{def1.5} \cite{Granas}
Let $\mathcal{A}: D(\mathcal{A})\subseteq X \rightarrow X$ be a bounded and continuous operator on a
Banach space $X$. Then $\mathcal{A}$ is called a condensing map if $\alpha(\mathcal{A}(B))<\alpha(B)$ for all bounded sets $B\subset D(\mathcal{A})$, where $\alpha$ denotes the Kuratowski measure of noncompactness.
\end{definition}

\begin{lemma} \label{lem1.3}  \emph{\cite[Example 11.7]{Zeidler}}
The map $\mathcal{K} + \mathcal{C}$ is a $k$-set contraction with $0 \leq k < 1$; and thus
also condensing, if
\begin{itemize}
  \item [(i)] $\mathcal{K},\ \mathcal{C} : D \subset X \rightarrow X$ are operators on the Banach space $X$;
  \item [(ii)] $\mathcal{K}$ is $k$-contractive, i.e.,
  \[\|\mathcal{K}x-\mathcal{K}y\|\leq\|x-y\|\]
  for all $x, y \in D$ and fixed $k \in [0, 1)$;
  \item [(iii)] $\mathcal{C}$ is compact.
\end{itemize}
\end{lemma}

\begin{theorem} \label{thm1.2}\cite{Sadovskii}
 Let $B$ be a convex, bounded and closed subset of a Banach space $X$ and
let $\mathcal{A}: B\rightarrow B$ be a condensing map. Then $\mathcal{A}$ has a fixed point.
\end{theorem}

Also, we present the Krasnoselskii-Schaefer's fixed point theorem due to Burton and Kirk \cite{Burton}, which is one of the key tools in our paper.

\begin{theorem} \label{thm1.3}\cite{Burton}
Let $X$ be a Banach space, and $A_{1}, A_{2} : X \rightarrow X$ be two operators such that $A_{1}$ is a contraction, and $A_{2}$ is completely continuous. Then, either
\begin{itemize}
 \item [(a)] the operator equation $u = A_{1}(u) + A_{2}(u)$ has a solution, or
  \item [(b)] the set $\Upsilon=\Big\{u\in X: \gamma A_{1}\big(\frac{u}{\gamma}\big)+\gamma A_{2}(u)=u\Big\}$ is unbounded for some $\gamma \in (0,1)$.
\end{itemize}
\end{theorem}

\begin{definition}\label{def1.6}
The $\psi$-Hilfer fractional boundary value problem \eqref{eq09}-\eqref{eq010} is said to be Ulam-Hyers stable if there exists a real number $c_{f}>0$ such that for each $\epsilon>0$ and $u\in \mathcal{C}(J,\mathbb{R})$ satisfying
\begin{equation} \label{eq.UH}
\Big|\big(^{H}D^{\nu,\beta;\psi}_{a^{+}}+\lambda ^{H}D^{\nu-1,\beta,\psi}_{a^{+}}\big)u(t)- f(t,u(t),(\mathcal{V}u)(t), I_{a^{+}}^{2-\mu,\psi}u(t))\Big|<\epsilon,\ \ \text{for all}\ t\in J,
\end{equation}
there exists a solution $x$ of problem \eqref{eq09}-\eqref{eq010} such that
\[|u(t)-x(t)|<c_{f}\epsilon \ \ \text{for all}\ t\in J.\]
\end{definition}

\begin{definition} \label{def1.7}
The $\psi$-Hilfer fractional boundary value problem \eqref{eq09}-\eqref{eq010} is called generalized Ulam-Hyers stable if there exists $\varphi_{f}\in C(\mathbb{R_{+}},\mathbb{R_{+}})$ with $ \varphi_{f}(0)=0$ such that for each $\epsilon>0$ and for each solution $u\in \mathcal{C}(J, \mathbb{R})$ of the inequality \eqref{eq.UH}, there exists a solution $x$ of problem \eqref{eq09}-\eqref{eq010} for which
$$ \big|u(t)-x(t)\big|\leq \varphi_{f}(\epsilon),\ t\in J.$$
\end{definition}

\begin{remark} \label{rem 0.2.1}
A function $u\in \mathcal{C}(J, \mathbb{R})$ is a solution of \eqref{eq.UH} if and only if there exists a function $z\in C(J, \mathbb{R})$ (which depends on $y$) such that
\begin{itemize}
  \item [(i)] $|z(t)|\leq \epsilon,\ t\in J,$
  \item [(ii)] $\big(^{H}D^{\nu,\beta;\psi}_{a^{+}}+\lambda ^{H}D^{\nu-1,\beta,\psi}_{a^{+}}\big)u(t)= f(t,u(t),(\mathcal{V}u)(t), I_{a^{+}}^{2-\mu,\psi}u(t))+z(t),\ t\in J.$
\end{itemize}
\end{remark}

For convenience, we denote
$$ \Delta=\psi_{a}(T)-\sum_{i=1}^{m}\frac{\psi_{a}^{\mu-2}(\eta_{i})}{\Gamma(\mu-1)}\bigg(\frac{\alpha_{i}\psi_{a}(\eta_{i})}{\mu-1}
+\beta_{i}\bigg)
,\ \
\mathcal{K}_{t}^{\nu}(s)=\frac{\psi^{\prime}(s)(\psi(t)-\psi(s))^{\nu-1}}{\Gamma(\nu)},$$
with $$\psi_{a}^{\varsigma-1}(t)=[\psi_{a}(t)]^{\varsigma-1}=(\psi(t)-\psi(a))^{\varsigma-1}.$$
\begin{lemma}\label{lem1}
Let $1<\nu\leq2$, $0\leq\beta<1$. Assume that $\Delta\neq 0$ and $ h \in C(J,\mathbb{R})$, then $u\in \mathcal{C}^{2}(J,\mathbb{R})$ is a solution of the sequential $\psi$-Hilfer fractional boundary value problem
\begin{equation}\label{eq1}
 \Big(^{H}D^{\nu,\beta;\psi}_{a^{+}}+\lambda ^{H}D^{\nu-1,\beta,\psi}_{a^{+}}\Big)u(t)=h(t),\  t \in J,
\end{equation}
\begin{equation}\label{eq2}
    u(a)=0,\  I^{2-\mu,\psi}_{a^{+}}u(T) =\sum_{i=1}^{m}\alpha_{i}u(\eta_{i})+\sum_{i=1}^{m}\beta_{i}u^{\prime}(\eta_{i})+g(u(\xi)),
 \end{equation}
if and only if $u$ satisfies the integral equation
\begin{eqnarray}\label{eq3}
u(t) &=&\int_{a}^{t}\mathcal{K}_{t}^{\nu}(s)h(s)ds+\frac{\psi_{a}^{\mu-1}(t)}{\Delta\Gamma(\mu)}
\Bigg[\sum_{i=1}^{m}\alpha_{i}\bigg(
\int_{a}^{\eta_{i}}\mathcal{K}_{\eta_{i}}^{\nu}(s)h(s)ds\bigg)\nonumber\\
&&+\sum_{i=1}^{m}\beta_{i}\psi(\eta_{i})\bigg(
\int_{a}^{\eta_{i}}\mathcal{K}_{\eta_{i}}^{\nu-1}(s)h(s)ds\bigg)-\int_{a}^{T}\mathcal{K}_{T}^{2-\mu+\nu}(s)h(s)ds+g(u(\xi))\Bigg]\nonumber\\
&&+\lambda \Bigg[\frac{\psi_{a}^{\mu-1}(t)}{\Delta\Gamma(\mu)}\Bigg(\int_{a}^{T}\mathcal{K}_{T}^{3-\nu}(s)u(s)ds      -\bigg(\sum_{i=1}^{m}\alpha_{i}\int_{a}^{\eta_{i}}\psi^{\prime}(s)u(s)ds \nonumber\\
&&+\sum_{i=1}^{m}\beta_{i}\psi^{\prime}(\eta_{i})
u(\eta_{i})\bigg)\Bigg)-\int_{a}^{t}\psi^{\prime}(s)u(s)ds \Bigg].
\end{eqnarray}
\end{lemma}

\begin{proof}
Applying $I^{\nu,\psi}_{a^{+}}$ to both sides of \eqref{eq1}, we get
\begin{equation}\label{eq3}
I^{\nu,\psi}_{a^{+}}D^{\nu,\beta,\psi}_{a^{+}}u(t)+\lambda I^{\nu,\psi}_{a^{+}}D^{\nu-1,\beta,\psi}_{a^{+}}u(t)=I^{\nu,\psi}_{a^{+}}h(t).
\end{equation}
By using Lemma \ref{lem }, and definitions \ref{def}, we may reduce \eqref{eq3} to an equivalent integral equation
\begin{eqnarray}\label{eq3}
u(t)  &=&I^{\nu,\psi}_{a^{+}}h(t)-\lambda I^{\nu,\psi}_{a^{+}}D^{\nu-1,\beta,\psi}_{a^{+}}u(t)
+\frac{\psi_{a}^{\mu-1}(t)}{\Gamma(\mu)}c_{1}
+\frac{\psi_{a}^{\mu-2}(t)}{\Gamma(\mu-1)}c_{2}\nonumber\\
&=&I^{\nu,\psi}_{a^{+}}h(t)-\lambda I^{1,\psi}_{a^{+}}I^{\nu-1,\psi}_{a^{+}}D^{\nu-1,\beta,\psi}_{a^{+}}u(t)
+\frac{\psi_{a}^{\mu-1}(t)}{\Gamma(\mu)}c_{1}
+\frac{\psi_{a}^{\mu-2}(t)}{\Gamma(\mu-1)}c_{2}
\end{eqnarray}
where $c_{1},c_{2}\in \mathbb{R}$ are arbitrary constants.\\
Put $\sigma=\nu-1$. Since $1<\nu\leq 2$, then $0<\sigma\leq1$,\ $\gamma=\sigma+\beta(1-\sigma)=\mu-1$. The equation \eqref{eq3} becomes
\begin{equation*}
u(t)=I^{\nu,\psi}_{a^{+}}h(t)-\lambda I^{1,\psi}_{a^{+}}I^{\sigma,\psi}_{a^{+}}D^{\sigma,\beta,\psi}_{a^{+}}u(t)
+\frac{\psi_{a}^{\mu-1}(t)}{\Gamma(\mu)}c_{1}
+\frac{\psi_{a}^{\mu-2}(t)}{\Gamma(\mu-1)}c_{2}
\end{equation*}
By Lemma \ref{lem}, we get
\begin{eqnarray*}
u(t)&=&I^{\nu,\psi}_{a^{+}}h(t)-\lambda I^{1,\psi}_{a^{+}}\bigg[u(t)-\frac{I^{1-\gamma,\psi}_{a^{+}}u(t)\big|_{t=a}}{\Gamma(\gamma)}\psi_{a}^{\gamma-1}(t)\bigg]
+\frac{\psi_{a}^{\mu-1}(t)}{\Gamma(\mu)}c_{1}
+\frac{\psi_{a}^{\mu-2}(t)}{\Gamma(\mu-1)}c_{2}\\
&=&I^{\nu,\psi}_{a^{+}}h(t)-\lambda I^{1,\psi}_{a^{+}}u(t)+\lambda\frac{I^{1-\gamma,\psi}_{a^{+}}u(t)\big|_{t=a}}{\Gamma(\gamma+1)}\psi_{a}^{\gamma}(t)
+\frac{\psi_{a}^{\mu-1}(t)}{\Gamma(\mu)}c_{1}
+\frac{\psi_{a}^{\mu-2}(t)}{\Gamma(\mu-1)}c_{2}\\
&=&I^{\nu,\psi}_{a^{+}}h(t)-\lambda I^{1,\psi}_{a^{+}}u(t)+\lambda\frac{I^{2-\mu,\psi}_{a^{+}}u(t)\big|_{t=a}}{\Gamma(\mu)}\psi_{a}^{\mu-1}(t)
+\frac{\psi_{a}^{\mu-1}(t)}{\Gamma(\mu)}c_{1}
+\frac{\psi_{a}^{\mu-2}(t)}{\Gamma(\mu-1)}c_{2}
\end{eqnarray*}
By the initial condition $u(a)=0$, we get $c_{2}=0$,$ \frac{I^{2-\mu,\psi}_{a^{+}}u(t)\big|_{t=a}}{\Gamma(\mu)}=0$ and hence
\begin{equation} \label{eq4}
u(t)  =I^{\nu,\psi}_{a^{+}}h(t)-\lambda I^{1,\psi}_{a^{+}}u(t)+\frac{\psi_{a}^{\mu-1}(t)}{\Gamma(\mu)}c_{1}.
\end{equation}
In view of Lemma \ref{ }, by taking the operator $ I^{2-\mu,\psi}_{0^{+}}$ into \eqref{eq4}, and differentiation \eqref{eq4}, we obtain
\begin{equation} \label{eqn5}
I^{2-\mu,\psi}_{a^{+}}u(t)=I^{2-\mu+\nu,\psi}_{a^{+}}h(t)-\lambda I^{3-\mu,\psi}_{a^{+}}u(t)+\psi_{a}(t)c_{1},
\end{equation}
\begin{equation} \label{eqn6}
u^{\prime}(t)=\psi(t)I^{\nu-1,\psi}_{a^{+}}h(t)-\lambda\psi^{^\prime}(t)u(t)
+\frac{\psi_{a}^{\mu-2}(t)}{\Gamma(\mu-1)}c_{1}.
\end{equation}
By the second boundary condition in \eqref{eq2}, together with \eqref{eqn5}-\eqref{eqn6}, we get
 \begin{equation*}
\begin{split}
I^{2-\mu+\nu,\psi}_{a^{+}}h(T)-\lambda I^{3-\mu,\psi}_{a^{+}}u(T)+\psi_{a}(T)c_{1}&=\sum_{i=1}^{m}\alpha_{i}u(\eta_{i})
+\sum_{i=1}^{m}\beta_{i}u^{\prime}(\eta_{i})+g(u(\xi))\\
&= \sum_{i=1}^{m}\alpha_{i}\Big(I^{\nu,\psi}_{a^{+}}h(\eta_{i})-\lambda I^{1,\psi}_{a^{+}}u(\eta_{i})+\frac{\psi_{a}^{\mu-1}(\eta_{i})}{\Gamma(\mu)}c_{1}\Big)\\
&\quad+\sum_{i=1}^{m}\beta_{i}\Big(\psi(\eta_{i})I^{\nu-1,\psi}_{a^{+}}h(\eta_{i})-\lambda\psi^{^\prime}
(\eta_{i})u(\eta_{i})\\
&\quad+\frac{\psi_{a}^{\mu-2}(\eta_{i})}{\Gamma(\mu-1)}c_{1} \Big)+g(u(\xi)),
\end{split}
\end{equation*}
from which we get
\begin{eqnarray*}
c_{1}&=&\frac{\lambda}{\Delta}\Bigg [I^{3-\mu,\psi}_{a^{+}}u(T)-\Bigg( \sum_{i=1}^{m}\alpha_{i}I^{1,\psi}_{a^{+}}u(\eta_{i})
+\sum_{i=1}^{m}\beta_{i}\psi^{\prime}(\eta_{i})u(\eta_{i}) \Bigg) \Bigg ]\\
&&+\frac{1}{\Delta}\Bigg [\sum_{i=1}^{m}\alpha_{i}I^{\nu,\psi}_{a^{+}}h(\eta_{i})
+\sum_{i=1}^{m}\beta_{i}\psi(\eta_{i})I^{\nu-1,\psi}_{a^{+}}h(\eta_{i})
-I^{2-\mu+\nu,\psi}_{a^{+}}h(T)+g(u(\xi))\Bigg]
\end{eqnarray*}
Substituting the values of $c_{1}$ in \eqref{eq4}, we get
\begin{eqnarray*}
u(t) & =&I_{a^{+}}^{\nu,\psi}h(t)+\lambda\frac{\psi_{a}^{\mu-1}(t)}{\Delta\Gamma(\mu)}\Bigg [I^{3-\mu,\psi}_{a^{+}}u(T)-\Bigg( \sum_{i=1}^{m}\alpha_{i}I^{1,\psi}_{a^{+}}u(\eta_{i})
+\sum_{i=1}^{m}\beta_{i}\psi^{\prime}(\eta_{i})u(\eta_{i}) \Bigg) \Bigg ]\\
&&+\frac{\psi_{a}^{\mu-1}(t)}{\Delta\Gamma(\mu)}\Bigg [\sum_{i=1}^{m}\alpha_{i}I^{\nu,\psi}_{a^{+}}h(\eta_{i})+\sum_{i=1}^{m}\beta_{i}\psi(\eta_{i})I^{\nu-1,\psi}_{a^{+}}h(\eta_{i})-I^{2-\mu+\nu,\psi}_{a^{+}}h(T)+g(u(\xi))\Bigg]\\
&&-\lambda I^{1,\psi}_{a^{+}}u(t)\\
&=&I_{a^{+}}^{\nu,\psi}h(t)+\frac{\psi_{a}^{\mu-1}(t)}{\Delta\Gamma(\mu)}\Bigg [\sum_{i=1}^{m}\alpha_{i}I^{\nu,\psi}_{a^{+}}h(\eta_{i})+\sum_{i=1}^{m}\beta_{i}\psi(\eta_{i})
I^{\nu-1,\psi}_{a^{+}}h(\eta_{i})-I^{2-\mu+\nu,\psi}_{a^{+}}h(T)\\
&&+g(u(\xi))\Bigg]+\lambda \Bigg[\frac{\psi_{a}^{\mu-1}(t)}{\Delta\Gamma(\mu)}\Bigg(I^{3-\mu,\psi}_{a^{+}}u(T)-\bigg(
\sum_{i=1}^{m}\alpha_{i}I^{1,\psi}_{a^{+}}u(\eta_{i})+\sum_{i=1}^{m}\beta_{i}\psi^{\prime}(\eta_{i})
u(\eta_{i})\bigg)\Bigg)\\
&&- I^{1,\psi}_{a^{+}}u(t)\Bigg]\\
& =&\int_{a}^{t}\mathcal{K}_{t}^{\nu}(s)h(s)ds+\frac{\psi_{a}^{\mu-1}(t)}{\Delta\Gamma(\mu)}
\Bigg[\sum_{i=1}^{m}\alpha_{i}\bigg(
\int_{a}^{\eta_{i}}\mathcal{K}_{\eta_{i}}^{\nu}(s)h(s)ds\bigg)\\
&&+\sum_{i=1}^{m}\beta_{i}\psi(\eta_{i})\bigg(
\int_{a}^{\eta_{i}}\mathcal{K}_{\eta_{i}}^{\nu-1}(s)h(s)ds\bigg)-\int_{a}^{T}\mathcal{K}_{T}^{2-\mu+\nu}(s)h(s)ds+g(u(\xi))\Bigg]\\
&&+\lambda \Bigg[\frac{\psi_{a}^{\mu-1}(t)}{\Delta\Gamma(\mu)}\Bigg(\int_{a}^{T}\mathcal{K}_{T}^{3-\nu}(s)u(s)ds      -\bigg(\sum_{i=1}^{m}\alpha_{i}\int_{a}^{\eta_{i}}\psi^{\prime}(s)u(s)ds+\sum_{i=1}^{m}\beta_{i}\psi^{\prime}(\eta_{i})
u(\eta_{i})\bigg)\Bigg)\\
&&-\int_{a}^{t}\psi^{\prime}(s)u(s)ds \Bigg]
\end{eqnarray*}
Conversely, we show that if $u$ satisfies the integral equation \eqref{eq3}, then it satisfies the problem \eqref{eq1}-\eqref{eq2}.
Applying the operator $D^{\nu,\beta,\psi}_{0^{+}}$ on both sides of the equation \eqref{eq3} and using Lemma \ref{Vanteler}, we obtain
\begin{equation*}
\begin{split}
D^{\nu,\beta,\psi}_{a^{+}}u(t)&=D^{\nu,\beta,\psi}_{a^{+}}I_{a^{+}}^{\nu,\psi}h(t)+\frac{\lambda}{\Delta\Gamma(\mu)}\Bigg [I^{3-\mu,\psi}_{a^{+}}u(T)-\Bigg( \sum_{i=1}^{m}\alpha_{i}I^{1,\psi}_{a^{+}}u(\eta_{i})+\sum_{i=1}^{m}\beta_{i}\psi^{\prime}(\eta_{i})u(\eta_{i}) \Bigg) \Bigg ]\\
&\quad \times D^{\nu,\beta,\psi}_{a^{+}}\psi_{a}^{\mu-1}(t)+\frac{1}
{\Delta\Gamma(\mu)}\Bigg[\sum_{i=1}^{m}\alpha_{i}I^{\nu,\psi}_{a^{+}}h(\eta_{i})
+\sum_{i=1}^{m}\beta_{i}\psi(\eta_{i})I^{\nu-1,\psi}_{a^{+}}h(\eta_{i})-I^{2-\mu+\nu,\psi}_{a^{+}}h(T)\\
&\quad+g(u(\xi))\Bigg]
D^{\nu,\beta,\psi}_{a^{+}}\psi_{a}^{\mu-1}(t)-\lambda D^{\nu,\beta,\psi}_{a^{+}}I_{a^{+}}^{1,\psi}u(t).
\end{split}
\end{equation*}

Since
$$ D^{\nu,\beta,\psi}_{a^{+}}\psi_{a}^{\mu-1}(t)=0,\ \text{for}\ \mu=\nu+\beta(2-\nu).$$
Then
\[\Big(D^{\nu,\beta,\psi}_{a^{+}}+\lambda D^{\nu-1,\beta,\psi}_{a^{+}}\Big)u(t)=h(t).\]
Further, it is easily to check $u(a)=0$, and
\begin{eqnarray*}
I_{a^{+}}^{2-\mu,\psi}u(T)&=&I_{a^{+}}^{2-\mu+\nu,\psi}h(T)+\frac{\psi_{a}(T)}{\Delta}\Bigg [\lambda\Bigg(I^{3-\mu,\psi}_{a^{+}}u(T)-\Big( \sum_{i=1}^{m}\alpha_{i}I^{1,\psi}_{a^{+}}u(\eta_{i})
+\sum_{i=1}^{m}\beta_{i}\psi^{\prime}(\eta_{i})u(\eta_{i}) \Big) \Bigg )\\
&&+\sum_{i=1}^{m}\alpha_{i}I^{\nu,\psi}_{a^{+}}h(\eta_{i})+\sum_{i=1}^{m}\beta_{i}\psi(\eta_{i})I^{\nu-1,\psi}_{a^{+}}h(\eta_{i})-I^{2-\mu+\nu,\psi}_{a^{+}}h(T)+g(u(\xi))\Bigg]\\
&&-\lambda I^{3-\mu,\psi}_{a^{+}}u(T).
\end{eqnarray*}

On the other hand, we have
\begin{eqnarray*}
\sum_{i=1}^{m}\alpha_{i}u(\eta_{i})&=&\sum_{i=1}^{m}\alpha_{i}I_{a^{+}}^{\nu,\psi}h(\eta_{i})
-\lambda\sum_{i=1}^{m}\alpha_{i}I_{a^{+}}^{1,\psi}u(\eta_{i})+\lambda\sum_{i=1}^{m}\alpha_{i}\frac{\psi_{a}^{\mu-1}(\eta_{i})}{\Delta\Gamma(\mu)}\\
&&\times\Bigg[I^{3-\mu,\psi}_{a^{+}}u(T)-\Big( \sum_{i=1}^{m}\alpha_{i}I^{1,\psi}_{a^{+}}u(\eta_{i})
+\sum_{i=1}^{m}\beta_{i}\psi^{\prime}(\eta_{i})u(\eta_{i})\Big)\Bigg]\\
&&+\sum_{i=1}^{m}\alpha_{i}\frac{\psi_{a}^{\mu-1}(\eta_{i})}{\Delta\Gamma(\mu)}\Bigg[\sum_{i=1}^{m}\alpha_{i}I^{\nu,\psi}_{a^{+}}h(\eta_{i})+\sum_{i=1}^{m}\beta_{i}\psi(\eta_{i})I^{\nu-1,\psi}_{a^{+}}h(\eta_{i})\\
&&-I^{2-\mu+\nu,\psi}_{a^{+}}h(T)+g(u(\xi)) \Bigg],
\end{eqnarray*}

\begin{eqnarray*}
\sum_{i=1}^{m}\beta_{i}u^{\prime}(\eta_{i})&=&\sum_{i=1}^{m}\beta_{i}\psi(\eta_{i})I_{a^{+}}^{\nu-1,\psi}h(\eta_{i})-\lambda \sum_{i=1}^{m}\beta_{i}\psi^{\prime}(\eta_{i})u(\eta_{i})+\lambda\sum_{i=1}^{m}\beta_{i}\frac{\psi_{a}^{\mu-2}(\eta_{i})}{\Delta\Gamma(\mu-1)}\\
&&\times\Bigg[I^{3-\mu,\psi}_{a^{+}}u(T)-\Big( \sum_{i=1}^{m}\alpha_{i}I^{1,\psi}_{a^{+}}u(\eta_{i})
+\sum_{i=1}^{m}\beta_{i}\psi^{\prime}(\eta_{i})u(\eta_{i})\Big) \Bigg]\\
&&+\sum_{i=1}^{m}\beta_{i}\frac{\psi_{a}^{\mu-2}(\eta_{i})}{\Delta\Gamma(\mu-1)}\Bigg[\sum_{i=1}^{m}\alpha_{i}I^{\nu,\psi}_{a^{+}}h(\eta_{i})+\sum_{i=1}^{m}\beta_{i}\psi(\eta_{i})I^{\nu-1,\psi}_{a^{+}}h(\eta_{i})\\
&&-I^{2-\mu+\nu,\psi}_{a^{+}}h(T)+g(u(\xi)) \Bigg],
\end{eqnarray*}

\begin{eqnarray*}
\sum_{i=1}^{m}\alpha_{i}u(\eta_{i})+\sum_{i=1}^{m}\beta_{i}u^{\prime}(\eta_{i})+g(u(\xi))&=&\sum_{i=1}^{m}\alpha_{i}I_{a^{+}}^{\nu,\psi}h(\eta_{i})
+\sum_{i=1}^{m}\beta_{i}\psi(\eta_{i})I_{a^{+}}^{\nu-1,\psi}h(\eta_{i})-\lambda \sum_{i=1}^{m}\alpha_{i}I_{a^{+}}^{1,\psi}u(\eta_{i})\\
&&-\lambda \sum_{i=1}^{m}\beta_{i}\psi^{\prime}(\eta_{i})u(\eta_{i})+g(u(\xi))+\lambda \frac{\psi_{a}(T)-\Delta}{\Delta}\\
&&\times\Bigg[I^{3-\mu,\psi}_{a^{+}}u(T)-\Big( \sum_{i=1}^{m}\alpha_{i}I^{1,\psi}_{a^{+}}u(\eta_{i})
+\sum_{i=1}^{m}\beta_{i}\psi^{\prime}(\eta_{i})u(\eta_{i})\Big)\Bigg]\\
&&+ \frac{\psi_{a}(T)-\Delta}{\Delta}\Bigg[\sum_{i=1}^{m}\alpha_{i}I^{\nu,\psi}_{a^{+}}h(\eta_{i})+\sum_{i=1}^{m}\beta_{i}\psi(\eta_{i})I^{\nu-1,\psi}_{a^{+}}h(\eta_{i})\\
&&-I^{2-\mu+\nu,\psi}_{a^{+}}h(T)+g(u(\xi)) \Bigg]\\
&=&I_{a^{+}}^{2-\mu+\nu,\psi}h(T)+\frac{\psi_{a}(T)}{\Delta}\Bigg [\lambda\Bigg(I^{3-\mu,\psi}_{a^{+}}u(T)-\Big( \sum_{i=1}^{m}\alpha_{i}I^{1,\psi}_{a^{+}}u(\eta_{i})\\
&&+\sum_{i=1}^{m}\beta_{i}\psi^{\prime}(\eta_{i})u(\eta_{i}) \Big) \Bigg )+\sum_{i=1}^{m}\alpha_{i}I^{\nu,\psi}_{a^{+}}h(\eta_{i})+\sum_{i=1}^{m}\beta_{i}\psi(\eta_{i})I^{\nu-1,\psi}_{a^{+}}h(\eta_{i})\\
&&-I^{2-\mu+\nu,\psi}_{a^{+}}h(T)+g(u(\xi))\Bigg]-\lambda I^{3-\mu,\psi}_{a^{+}}u(T)\\
&=&I_{a^{+}}^{2-\mu,\psi}u(T).
\end{eqnarray*}

Hence, we get \eqref{eq1}-\eqref{eq2}.
\end{proof}

\section{Existence results}
In view of Lemma \ref{lem1}, we define the operator $\mathcal{A}:\mathcal{C}(J,\mathbb{R})\rightarrow \mathcal{C}(J,\mathbb{R})$ by
\begin{eqnarray}
(Au)(t) &=&\int_{a}^{t}\mathcal{K}_{t}^{\nu}(s)F_{u}(s)ds+\frac{\psi_{a}^{\mu-1}(t)}{\Delta\Gamma(\mu)}
\Bigg[\sum_{i=1}^{m}\alpha_{i}\bigg(
\int_{a}^{\eta_{i}}\mathcal{K}_{\eta_{i}}^{\nu}(s)F_{u}(s)ds\bigg)\nonumber\\
&&+\sum_{i=1}^{m}\beta_{i}\psi(\eta_{i})\bigg(
\int_{a}^{\eta_{i}}\mathcal{K}_{\eta_{i}}^{\nu-1}(s)F_{u}(s)ds\bigg)-\int_{a}^{T}\mathcal{K}_{T}^{2-\mu+\nu}(s)
F_{u}(s)ds+g(u(\xi))\Bigg]\nonumber\\
&&+\lambda \Bigg[\frac{\psi_{a}^{\mu-1}(t)}{\Delta\Gamma(\mu)}\Bigg(\int_{a}^{T}\mathcal{K}_{T}^{3-\nu}(s)u(s)ds      -\bigg(\sum_{i=1}^{m}\alpha_{i}\int_{a}^{\eta_{i}}\psi^{\prime}(s)u(s)ds \nonumber\\
&&+\sum_{i=1}^{m}\beta_{i}\psi^{\prime}(\eta_{i})
u(\eta_{i})\bigg)\Bigg)-\int_{a}^{t}\psi^{\prime}(s)u(s)ds \Bigg] \label{eqn7}.
\end{eqnarray}

Specifically, it can be expressed as
{\color{blue}\begin{eqnarray}  \label{eqn8}
 (\mathcal{A}u)(t)&=&(\mathcal{A}_{1}u)(t)+(\mathcal{A}_{2}u)(t),\ t\in J,
\end{eqnarray}}
where
{\color{blue}\begin{eqnarray}
(\mathcal{A}_{1}u)(t)&=&I_{a^{+}}^{\nu,\psi}F_{u}(t)+\frac{\psi_{a}^{\mu-1}(t)}{\Delta\Gamma(\mu)}\Bigg [\sum_{i=1}^{m}\alpha_{i}I^{\nu,\psi}_{a^{+}}F_{u}(\eta_{i})+\sum_{i=1}^{m}\beta_{i}\psi(\eta_{i})
I^{\nu-1,\psi}_{a^{+}}F_{u}(\eta_{i})\nonumber\\
&&-I^{2-\mu+\nu,\psi}_{a^{+}}F_{u}(T) \Bigg]
+\lambda \Bigg[\frac{\psi_{a}^{\mu-1}(t)}{\Delta\Gamma(\mu)}\Bigg(I^{3-\mu,\psi}_{a^{+}}u(T)-\bigg(
\sum_{i=1}^{m}\alpha_{i}I^{1,\psi}_{a^{+}}u(\eta_{i})\nonumber\\
&&+\sum_{i=1}^{m}\beta_{i}\psi^{\prime}(\eta_{i})
u(\eta_{i})\bigg)\Bigg)- I^{1,\psi}_{a^{+}}u(t)\Bigg]\label{eqn8},\\
(\mathcal{A}_{2}u)(t)&=&\frac{\psi_{a}^{\mu-1}(t)}{\Delta\Gamma(\mu)}g(u(\xi)) \label{eqn9},
\end{eqnarray}}
\[F_{u}(t)=f(t,u(t),(\mathcal{V}u)(t), I_{a^{+}}^{2-\mu,\psi}u(t)),\  t\in J.\]
Clearly, the operator $\mathcal{A}$ has fixed points if and only if the $\psi$-Hilfer fractional boundary value problem \eqref{eq09}-\eqref{eq010} has solutions.\\

We now list suitable assumptions on the nonlinearity function $f$.
  \begin{itemize}
\item[($\mathcal{H}_{1}$)] $f\in \mathcal{C}(J \times \mathbb{R}^{3},\mathbb{R})$, and $f(t,0,0,0)\not\equiv0$ on $J$.
\item[($\mathcal{H}_{2}$)] There exist three nonnegative functions $l_{1}, l_{2}, l_{3} \in \mathcal{C}(J,[0,\infty))$ such that for all
$t\in J$ and $u_{1}, u_{2}, v_{1}, v_{2}, w_{1}, w_{2}\in \mathbb{R}$, we have
    \[|f(t,u_{1},v_{1},w_{1})-f(t,u_{2},v_{2},w_{2})|\leq l_{1}(t)|u_{1}-u_{2}|+l_{2}(t)|v_{1}-v_{2}|+l_{3}(t)|w_{1}-w_{2}|,\]
  \item[($\mathcal{H}_{3}$)]  $g$ is a continuous function satisfying $g(a)=0$ and
      \[|g(u_{1})-g(u_{2})|\leq N|u_{1}-u_{2}|,\ \forall u_{1}, u_{2}\in \mathbb{R},\ N>0.\]
\item[($\mathcal{H}_{4}$)] $\mathcal{V}:\mathcal{C}( J, \mathbb{R})\rightarrow \mathcal{C}( J, \mathbb{R})$ (not necessarily linear) satisfying
    \begin{enumerate}
         \item $\|\mathcal{V}u\|\leq \|u\|$,
      \item  $\|\mathcal{V}u-\mathcal{V}v\|\leq \|u-v\|$.
    \end{enumerate}

    \end{itemize}

For convenience, we get the following notations:
$$\Theta=\frac{\psi_{a}^{\nu}(T)}{\Gamma(\nu+1)}+\frac{\psi_{a}^{\mu-1}(T)}{|\Delta|\Gamma(\mu)
}\Bigg [\sum_{i=1}^{m}|\alpha_{i}|\frac{\psi_{a}^{\nu}(\eta_{i})}{\Gamma(\nu+1)}
+\sum_{i=1}^{m}|\beta_{i}|\psi(\eta_{i})\frac{\psi_{a}^{\nu-1}(\eta_{i})}{\Gamma(\nu)}+
\frac{\psi_{a}^{2-\mu+\nu}(T)}{\Gamma(3-\mu+\nu)} \Bigg],$$
 $$\Phi=\frac{\psi_{a}^{\mu-1}(T)}{|\Delta|\Gamma(\mu)}\Bigg(
\frac{\psi_{a}^{3-\mu}(T)}{\Gamma(4-\mu)}+\sum_{i=1}^{m}|\alpha_{i}|\psi_{a}(\eta_{i})+
\sum_{i=1}^{m}|\beta_{i}|\psi^{\prime}(\eta_{i})\Bigg)+\psi_{a}(T) ,$$
$$\Lambda(t,\mu)=l_{1}^{\star}+l_{2}^{\star}+l_{3}^{\star}\frac{\psi_{a}^{2-\mu}(t)}{\Gamma(3-\mu)},\ l_{i}^{\star}=\max_{t\in J}l_{i}(t),\ i\in \{1,2,3\},$$
 $$\Omega=\Lambda(T,\mu)\Theta+|\lambda|\Phi.$$

Our first result on the existence and uniqueness of solutions is based on the Banach contraction (Banach’s contraction mapping principle) principle.\\
\begin{theorem} \label{thmm1}
Assume that $(\rm{\mathcal{H}_{1}})$-$(\rm{\mathcal{H}_{4}})$ hold. If
 \begin{equation} \label{eq6}
\Omega+\frac{\psi_{a}^{\mu-1}(T)}{|\Delta|\Gamma(\mu)}N<1,
\end{equation}
then the $\psi$-Hilfer fractional boundary value problem \eqref{eq09}-\eqref{eq010} has a unique solution on $J$.
\end{theorem}
\begin{proof}
Setting $\sup\big\{|f(t,0,0,0)|,\ t\in J \big\}=L<\infty$, and choose a constant $r>0$ with
$$r\geq\frac{L\Theta+\frac{\psi_{a}^{\mu-1}(T)}{|\Delta|\Gamma(\mu)}Na}{1-\Big(\Omega+\frac{\psi_{a}^{\mu-1}(T)}{|\Delta|\Gamma(\mu)}N\Big)}.$$
Define a bounded, closed and convex set
$$\mathcal{B}_{r}=\big\{u\in\mathcal{C}(J,\mathbb{R}):\|u\|\leq r\big\}.$$
Next, we prove that $\mathcal{A}$ is a contraction map on $\mathcal{C}(J,\mathbb{R})$ with respect to the norm $\|.\|$.
First, we show that $\mathcal{A}(\mathcal{B}_{r})\subset \mathcal{B}_{r}$.
By Lemmas \ref{lem1.3}-\ref{lem1.4}, for any $u\in \mathcal{B}_{r}$ and $t\in [0,1]$, using $\rm (\mathcal{H}_{2})$, we have
\begin{eqnarray*}
|F_{u}(t)|&=&|f(t,u(t),(\mathcal{V}u)(t), I_{a^{+}}^{2-\mu,\psi}u(t))|\\
&\leq& |f(t,u(t),(\mathcal{V}u)(t), I_{a^{+}}^{2-\mu,\psi}u(t))-f(t,0,0,0)|+|f(t,0,0,0)|\\
&\leq& l_{1}(t)|u(t)|+l_{2}(t)|(\mathcal{V}u)(t)|+l_{3}(t) I_{a^{+}}^{2-\mu,\psi}|u(t)|+L\\
&\leq& l_{1}^{\star}\|u\|+l_{2}^{\star}\|(\mathcal{V}u)(t)\|+l_{3}^{\star} \frac{\psi_{a}^{2-\mu}(t)}{\Gamma(3-\mu)}\|u\|+L\\
&\leq& \Lambda(t,\mu)\|u\|+L
\end{eqnarray*}

 \begin{eqnarray*}
 \big|(\mathcal{A}u)(t)\big|&\leq& |(\mathcal{A}_{1}u)(t)|+|(\mathcal{A}_{2}u)(t)|\\
  &\leq& I_{a^{+}}^{\nu,\psi}|F_{u}(t)|+\frac{\psi_{a}^{\mu-1}(t)}{|\Delta|\Gamma(\mu)}\Bigg [\sum_{i=1}^{m}|\alpha_{i}|I^{\nu,\psi}_{a^{+}}|F_{u}(\eta_{i})|+\sum_{i=1}^{m}|\beta_{i}|\psi(\eta_{i})
I^{\nu-1,\psi}_{a^{+}}|F_{u}(\eta_{i})|+I^{2-\mu+\nu,\psi}_{a^{+}}|F_{u}(T)| \Bigg]\nonumber\\
&&+|\lambda| \Bigg[\frac{\psi_{a}^{\mu-1}(t)}{\Delta\Gamma(\mu)}\Bigg(I^{3-\mu,\psi}_{a^{+}}|u(T)|+
\sum_{i=1}^{m}|\alpha_{i}|I^{1,\psi}_{a^{+}}|u(\eta_{i})|+\sum_{i=1}^{m}|\beta_{i}|\psi^{\prime}(\eta_{i})
|u(\eta_{i})|\Bigg)+ I^{1,\psi}_{a^{+}}|u(t)|\Bigg]\\
&&+\frac{\psi_{a}^{\mu-1}(t)}{|\Delta|\Gamma(\mu)}|g(u(\xi))|\\
&\leq& \big(\Lambda(T,\mu)\|u\|+L\big)\Bigg\{ \int_{a}^{t}\frac{\psi^{\prime}(s)(\psi(t)-\psi(s))^{\nu-1}}{\Gamma(\nu)}ds
+\frac{\psi_{a}^{\mu-1}(t)}{|\Delta|\Gamma(\mu)}\\
&&\times\Bigg[\sum_{i=1}^{m}|\alpha_{i}|\int_{a}^{\eta_{i}}\frac{\psi^{\prime}(s)
(\psi(\eta_{i})-\psi(s))^{\nu-1}}{\Gamma(\nu)}ds+\sum_{i=1}^{m}|\beta_{i}|\psi(\eta_{i})\int_{a}^{\eta_{i}}\frac{\psi^{\prime}(s)
(\psi(\eta_{i})-\psi(s))^{\nu-2}}{\Gamma(\nu-1)}ds\\
&&+\int_{a}^{T}\frac{\psi^{\prime}(s)
(\psi(T)-\psi(s))^{1-\mu+\nu}}{\Gamma(2-\mu+\nu)}ds\Bigg]\Bigg\}\\
&&+|\lambda|    \|u\|\Bigg[\frac{\psi_{a}^{\mu-1}(t)}{|\Delta|\Gamma(\mu)}\Bigg(\int_{a}^{T}\frac{\psi^{\prime}(s)(\psi(T)-\psi(s))^{2-\mu}}
{\Gamma(3-\mu)}ds+
\sum_{i=1}^{m}|\alpha_{i}|\int_{a}^{\eta_{i}}\psi^{\prime}(s)ds+\sum_{i=1}^{m}|\beta_{i}|\psi^{\prime}(\eta_{i})\Bigg)\\
&&+\int_{a}^{t}\psi^{\prime}(s)ds\bigg)\Bigg]+N\frac{\psi_{a}^{\mu-1}(t)}{|\Delta|\Gamma(\mu)}\big(\|u\|+a\big)\nonumber\\
&\leq& \big(\Lambda(T,\mu)\|u\|+L\big)\Bigg\{\frac{\psi_{a}^{\nu}(T)}{\Gamma(\nu+1)}+\frac{\psi_{a}^{\mu-1}(T)}{|\Delta|\Gamma(\mu)
}\Bigg [\sum_{i=1}^{m}|\alpha_{i}|\frac{\psi_{a}^{\nu}(\eta_{i})}{\Gamma(\nu+1)}
+\sum_{i=1}^{m}|\beta_{i}|\psi(\eta_{i})\frac{\psi_{a}^{\nu-1}(\eta_{i})}{\Gamma(\nu)}\\
&&+\frac{\psi_{a}^{2-\mu+\nu}(T)}{\Gamma(3-\mu+\nu)} \Bigg]\Bigg\}
+|\lambda|  \|u\|\Bigg[\frac{\psi_{a}^{\mu-1}(T)}{|\Delta|\Gamma(\mu)}\Bigg(
\frac{\psi_{a}^{3-\mu}(T)}{\Gamma(4-\mu)}+\sum_{i=1}^{m}|\alpha_{i}|\psi_{a}(\eta_{i})+
\sum_{i=1}^{m}|\beta_{i}|\psi^{\prime}(\eta_{i})\Bigg)\\
&&+\psi_{a}(T) \Bigg]+N\frac{\psi_{a}^{\mu-1}(T)}{|\Delta|\Gamma(\mu)}\big(\|u\|+a\big)\\
&\leq& \Omega r+L\Theta+N\frac{\psi_{a}^{\mu-1}(T)}{|\Delta|\Gamma(\mu)}(r+a)\leq r,
  \end{eqnarray*}
In view of \eqref{eq6}, we get $\|\mathcal{A}u\|\leq r$.  Then $\mathcal{A}$ maps $\mathcal{B}_{r}$ into itself.
From $\rm (\mathcal{H}_{2})$, it follows that
 \begin{eqnarray*}
 \big|F_{u}(t)-F_{v}(t)\big|&\leq& l_{1}(t)|u(t)-v(t)|+l_{2}(t)|(\mathcal{V}u)(t)-(\mathcal{V}u)(t)|+l_{3}(t)I_{a^{+}}^{2-\mu,\psi}|u(t)-v(t)|\\
 &\leq& l_{1}^{\star}\|u-v\|+l_{2}^{\star}\|u-v\|+l_{3}^{\star} \frac{\psi_{a}^{2-\mu}(t)}{\Gamma(3-\mu)} \|u-v\|\\
 &\leq& \Lambda(t,\mu) \|u-v\|.
 \end{eqnarray*}

 Then, for all $u,v \in\mathcal{C}(J,\mathbb{R})$ and for each $t\in J$, using $\rm (\mathcal{H}_{2})$, we have
 \begin{eqnarray*}
 \big|(\mathcal{A}u)(t)-(\mathcal{A}v)(t)|&\leq&|(\mathcal{A}_{1}u)(t)|+|(\mathcal{A}_{2}u)(t)|\\
  &\leq& I_{a^{+}}^{\nu,\psi}|F_{u}(t)-F_{v}(t)|+\frac{\psi_{a}^{\mu-1}(t)}{|\Delta|\Gamma(\mu)}\Bigg [\sum_{i=1}^{m}|\alpha_{i}|I^{\nu,\psi}_{a^{+}}|F_{u}(\eta_{i})-F_{v}(\eta_{i})|\\
  &&+\sum_{i=1}^{m}|\beta_{i}|\psi(\eta_{i})
I^{\nu-1,\psi}_{a^{+}}|F_{u}(\eta_{i})-F_{v}(\eta_{i})|+I^{2-\mu+\nu,\psi}_{a^{+}}|F_{u}(T)-F_{v}(T)| \Bigg]\nonumber\\
&&+|\lambda| \Bigg[\frac{\psi_{a}^{\mu-1}(t)}{\Delta\Gamma(\mu)}\Bigg(I^{3-\mu,\psi}_{a^{+}}|u(T)-v(T)|+
\sum_{i=1}^{m}|\alpha_{i}|I^{1,\psi}_{a^{+}}|u(\eta_{i})-v(\eta_{i})|\\
&&+\sum_{i=1}^{m}|\beta_{i}|\psi^{\prime}(\eta_{i})
|u(\eta_{i})-v(\eta_{i})|\bigg)+ I^{1,\psi}_{a^{+}}|u(t)-v(t)|\Bigg]\\
&&+\frac{\psi_{a}^{\mu-1}(t)}{|\Delta|\Gamma(\mu)}|g(u(\xi))-g(v(\xi))|\\
&\leq& \Bigg\{\Lambda(T,\mu)\Bigg(\frac{\psi_{a}^{\nu}(T)}{\Gamma(\nu+1)}+\frac{\psi_{a}^{\mu-1}(T)}{|\Delta|\Gamma(\mu)
}\Bigg [\sum_{i=1}^{m}|\alpha_{i}|\frac{\psi_{a}^{\nu}(\eta_{i})}{\Gamma(\nu+1)}
+\sum_{i=1}^{m}|\beta_{i}|\psi(\eta_{i})\frac{\psi_{a}^{\nu-1}(\eta_{i})}{\Gamma(\nu)}\\
&&+\frac{\psi_{a}^{2-\mu+\nu}(T)}{\Gamma(3-\mu+\nu)} \Bigg]\Bigg)
+|\lambda| \Bigg[\frac{\psi_{a}^{\mu-1}(T)}{|\Delta|\Gamma(\mu)}\Bigg(
\frac{\psi_{a}^{3-\mu}(T)}{\Gamma(4-\mu)}+\sum_{i=1}^{m}|\alpha_{i}|\psi_{a}(\eta_{i})+
\sum_{i=1}^{m}|\beta_{i}|\psi^{\prime}(\eta_{i})\Bigg)\\
&&+\psi_{a}(T) \Bigg]\Bigg\}\|u-v\|+N\frac{\psi_{a}^{\mu-1}(T)}{|\Delta|\Gamma(\mu)}\|u-v\|\\
&\leq& \Bigg(\Omega +N\frac{\psi_{a}^{\mu-1}(T)}{|\Delta|\Gamma(\mu)}\Bigg)\|u-v\|
   \end{eqnarray*}
   Thus we obtain
  $$ \|\mathcal{ A}u-\mathcal{A}v\|\leq \Bigg(\Omega +N\frac{\psi_{a}^{\mu-1}(T)}{|\Delta|\Gamma(\mu)}\Bigg) \|u-v\|.$$

From \eqref{eq6}, it follows that $\mathcal{A}$ is a contraction map. Therefore, by the Banach contraction principle, $\mathcal{A}$ has a unique fixed point which is the unique solution of the problem \eqref{eq09}-\eqref{eq010} in $\mathcal{C}(J,\mathbb{R})$.
\end{proof}

The existence result is based on Sadovskii's fixed point theorem.

\begin{theorem}\label{thmm2}
Assume that $(\rm{\mathcal{H}_{1}})$, $(\rm{\mathcal{H}_{3}})$ and $(\rm{\mathcal{H}_{4}})$ hold. In addition, we assume that
\begin{itemize}
\item[($\mathcal{H}_{5}$)] There exist nonnegative functions $p_{i} \in C(J, \mathbb{R})$ and nondecreasing functions $\varphi_{i} :\mathbb{R}^{+}\rightarrow \mathbb{R}^{+},\ i =1,2,3,$ such that
    \[\big|f(t,x_{1},x_{2},x_{3})\big|\leq \sum_{i=1}^{3}p_{i}(t)\varphi_{i}(|x_{i}|),\ \text{for all}\ (t,x_{1},x_{2},x_{3})\in J\times \mathbb{R}\times\mathbb{R}\times\mathbb{R},\] with $p_{i}^{\star}=\max_{t\in J}|p_{i}(t)|,\ i=1,2,3.$
   \item[($\mathcal{H}_{6}$)]  $\Xi:=|\lambda|\Phi +N\frac{\psi_{a}^{\mu-1}(T)}{|\Delta|\Gamma(\mu)}<1.$
  \item[($\mathcal{H}_{7}$)] There exists a number $r_{0}>0$ such that \[ r_{0}> \frac{\Big(\sum_{i=1}^{2}p_{i}^{\star}\varphi_{i}(r_{0})+p_{3}^{\star}\varphi_{3}\Big(\frac{\psi_{a}^{2-\mu}(T)}{\Gamma(3-\mu)}r_{0}\Big) \Big)\Theta+\frac{\psi_{a}^{\mu-1}(T)}{|\Delta|\Gamma(\mu)}Na}{1-\Xi}.\]
\end{itemize}
Then the $\psi$-Hilfer fractional boundary value problem \eqref{eq09}-\eqref{eq010} has at least one solution on $J$.
\end{theorem}
\begin{proof}
Let us consider the operators $\mathcal{A}_{1}$, $\mathcal{A}_{2}$ defined by \eqref{eqn8} and \eqref{eqn9} respectively. The proof is divided into several steps:\\
\textbf{Step 1.} Let $\mathcal{B}_{r_{0}}=\big\{u\in\mathcal{C}(J,\mathbb{R}):\|u\|\leq r_{0}\big\}$ be a closed bounded and convex subset of $\mathcal{C}(J,\mathbb{R})$, where $r_{0}$ is defined by assumption $(\rm{\mathcal{H}_{7}})$. We first show that  $\mathcal{A}\big(\mathcal{B}_{r_{0}}\big)\subset \mathcal{B}_{r_{0}}$.\\
For any $u\in \mathcal{B}_{r_{0}}$, $t\in J$, by $(\rm{\mathcal{H}_{5}})$-$(\rm{\mathcal{H}_{7}})$, we get

\begin{eqnarray*}
 \big|(\mathcal{A}u)(t)\big|&\leq& |(\mathcal{A}_{1}u)(t)|+|(\mathcal{A}_{2}u)(t)|\\
  &\leq& I_{a^{+}}^{\nu,\psi}|F_{u}(t)|+\frac{\psi_{a}^{\mu-1}(t)}{|\Delta|\Gamma(\mu)}\Bigg [\sum_{i=1}^{m}|\alpha_{i}|I^{\nu,\psi}_{a^{+}}|F_{u}(\eta_{i})|+\sum_{i=1}^{m}|\beta_{i}|\psi(\eta_{i})
I^{\nu-1,\psi}_{a^{+}}|F_{u}(\eta_{i})|+I^{2-\mu+\nu,\psi}_{a^{+}}|F_{u}(T)| \Bigg]\nonumber\\
&&+|\lambda| \Bigg[\frac{\psi_{a}^{\mu-1}(t)}{|\Delta|\Gamma(\mu)}\Bigg(I^{3-\mu,\psi}_{a^{+}}|u(T)|+
\sum_{i=1}^{m}|\alpha_{i}|I^{1,\psi}_{a^{+}}|u(\eta_{i})|+\sum_{i=1}^{m}|\beta_{i}|\psi^{\prime}(\eta_{i})
|u(\eta_{i})|\Bigg)+ I^{1,\psi}_{a^{+}}|u(t)|\Bigg]\\
&&+\frac{\psi_{a}^{\mu-1}(t)}{|\Delta|\Gamma(\mu)}|g(u(\xi))|\\
&\leq& \Bigg(\sum_{i=1}^{2}p_{i}^{\star}\varphi_{i}(\|u\|)+p_{3}^{\star}\varphi_{3}\bigg(\frac{\psi_{a}^{2-\mu}(T)}{\Gamma(3-\mu)}\|u\|\bigg) \Bigg)\Bigg\{ \int_{a}^{t}\frac{\psi^{\prime}(s)(\psi(t)-\psi(s))^{\nu-1}}{\Gamma(\nu)}ds
+\frac{\psi_{a}^{\mu-1}(t)}{|\Delta|\Gamma(\mu)}\\
&&\times\Bigg[\sum_{i=1}^{m}|\alpha_{i}|\int_{a}^{\eta_{i}}\frac{\psi^{\prime}(s)
(\psi(\eta_{i})-\psi(s))^{\nu-1}}{\Gamma(\nu)}ds+\sum_{i=1}^{m}|\beta_{i}|\psi(\eta_{i})\int_{a}^{\eta_{i}}\frac{\psi^{\prime}(s)
(\psi(\eta_{i})-\psi(s))^{\nu-2}}{\Gamma(\nu-1)}ds\\
&&+\int_{a}^{T}\frac{\psi^{\prime}(s)
(\psi(T)-\psi(s))^{1-\mu+\nu}}{\Gamma(2-\mu+\nu)}ds\Bigg]\Bigg\}\\
&&+|\lambda|    \|u\|\Bigg[\frac{\psi_{a}^{\mu-1}(t)}{|\Delta|\Gamma(\mu)}\Bigg(\int_{a}^{T}\frac{\psi^{\prime}(s)(\psi(T)-\psi(s))^{2-\mu}}
{\Gamma(3-\mu)}ds+
\sum_{i=1}^{m}|\alpha_{i}|\int_{a}^{\eta_{i}}\psi^{\prime}(s)ds+\sum_{i=1}^{m}|\beta_{i}|\psi^{\prime}(\eta_{i})\Bigg)\\
&&+\int_{a}^{t}\psi^{\prime}(s)ds\bigg)\Bigg]+N\frac{\psi_{a}^{\mu-1}(t)}{|\Delta|\Gamma(\mu)}\big(\|u\|+a\big)\nonumber\\
&\leq& \Bigg(\sum_{i=1}^{2}p_{i}^{\star}\varphi_{i}(r_{0})+p_{3}^{\star}\varphi_{3}\bigg(\frac{\psi_{a}^{2-\mu}(T)}{\Gamma(3-\mu)}r_{0}\bigg) \Bigg)\Bigg\{\frac{\psi_{a}^{\nu}(T)}{\Gamma(\nu+1)}+\frac{\psi_{a}^{\mu-1}(T)}{|\Delta|\Gamma(\mu)
}\Bigg [\sum_{i=1}^{m}|\alpha_{i}|\frac{\psi_{a}^{\nu}(\eta_{i})}{\Gamma(\nu+1)}\\
&&+\sum_{i=1}^{m}|\beta_{i}|\psi(\eta_{i})\frac{\psi_{a}^{\nu-1}(\eta_{i})}{\Gamma(\nu)}
+\frac{\psi_{a}^{2-\mu+\nu}(T)}{\Gamma(3-\mu+\nu)} \Bigg]\Bigg\}
+|\lambda|  r_{0}\Bigg[\frac{\psi_{a}^{\mu-1}(T)}{|\Delta|\Gamma(\mu)}\Bigg(
\frac{\psi_{a}^{3-\mu}(T)}{\Gamma(4-\mu)}+\sum_{i=1}^{m}|\alpha_{i}|\psi_{a}(\eta_{i})\\
&&+\sum_{i=1}^{m}|\beta_{i}|\psi^{\prime}(\eta_{i})\Bigg)
+\psi_{a}(T) \Bigg]+N\frac{\psi_{a}^{\mu-1}(T)}{|\Delta|\Gamma(\mu)}\big(r_{0}+a\big)\\
&\leq&  \Bigg(\sum_{i=1}^{2}p_{i}^{\star}\varphi_{i}(r_{0})+p_{3}^{\star}\varphi_{3}\bigg(\frac{\psi_{a}^{2-\mu}(T)}{\Gamma(3-\mu)}r_{0}\bigg) \Bigg)\Theta+\Xi r_{0}+\frac{\psi_{a}^{\mu-1}(T)}{|\Delta|\Gamma(\mu)}Na<r_{0},
  \end{eqnarray*}
which implies that $\mathcal{A}\big(\mathcal{B}_{r_{0}}\big)\subset \mathcal{B}_{r_{0}}$.\\
\textbf{Step 2.} $\mathcal{A}_{1}$ is compact.
In view of Step 1, the operator $\mathcal{A}_{1}$ is uniformly bounded. Now we show that $\mathcal{A}_{1}$ maps
bounded sets into equicontinuous sets of $\mathcal{C}(J,\mathbb{R})$. Let $t_{1}, t_{2}\in J$ with $t_{1}< t_{2}$ and
$u\in \mathcal{B}_{r_{0}}$. Then we have

\begin{eqnarray*}
\big|(\mathcal{A}_{1}u)(t_{2})|-|(\mathcal{A}_{1}u)(t_{1})\big|
  &\leq&  \Big|I_{a^{+}}^{\nu,\psi}F_{u}(t_{2})-I_{a^{+}}^{\nu,\psi}F_{u}(t_{1})\Big|+\frac{\psi_{a}^{\mu-1}(t_{2})-\psi_{a}^{\mu-1}(t_{1})}{|\Delta|\Gamma(\mu)}\Bigg [\sum_{i=1}^{m}|\alpha_{i}|I^{\nu,\psi}_{a^{+}}|F_{u}(\eta_{i})|\\
  &&+\sum_{i=1}^{m}|\beta_{i}|\psi(\eta_{i})
I^{\nu-1,\psi}_{a^{+}}|F_{u}(\eta_{i})|+I^{2-\mu+\nu,\psi}_{a^{+}}|F_{u}(T)| \Bigg]\nonumber\\
&&+|\lambda| \Bigg[\frac{\psi_{a}^{\mu-1}(t_{2})-\psi_{a}^{\mu-1}(t_{1})}{|\Delta|\Gamma(\mu)}\Bigg(I^{3-\mu,\psi}_{a^{+}}|u(T)|+
\sum_{i=1}^{m}|\alpha_{i}|I^{1,\psi}_{a^{+}}|u(\eta_{i})|\\
&&+\sum_{i=1}^{m}|\beta_{i}|\psi^{\prime}(\eta_{i})
|u(\eta_{i})|\Bigg)+ \big|I^{1,\psi}_{a^{+}}|u(t_{2})-I^{1,\psi}_{a^{+}}|u(t_{1})|\Bigg]\\
&\leq&
\Bigg(\sum_{i=1}^{2}p_{i}^{\star}\varphi_{i}(r_{0})+p_{3}^{\star}\varphi_{3}\bigg(\frac{\psi_{a}^{2-\mu}(T)}{\Gamma(3-\mu)}r_{0}\bigg) \Bigg)\Bigg\{\frac{\psi_{a}^{\nu}(t_{2})-\psi_{a}^{\nu}(t_{1})}{\Gamma(\nu+1)}+\frac{\psi_{a}^{\mu-1}(t_{2})-\psi_{a}^{\mu-1}(t_{1})}{|\Delta|\Gamma(\mu)
}\\
&&\times\Bigg [\sum_{i=1}^{m}|\alpha_{i}|\frac{\psi_{a}^{\nu}(\eta_{i})}{\Gamma(\nu+1)}
+\sum_{i=1}^{m}|\beta_{i}|\psi(\eta_{i})\frac{\psi_{a}^{\nu-1}(\eta_{i})}{\Gamma(\nu)}
+\frac{\psi_{a}^{2-\mu+\nu}(T)}{\Gamma(3-\mu+\nu)} \Bigg]\Bigg\}\\
&&+|\lambda|  r_{0}\Bigg[\frac{\psi_{a}^{\mu-1}(t_{2})-\psi_{a}^{\mu-1}(t_{1})}{|\Delta|\Gamma(\mu)}\Bigg(
\frac{\psi_{a}^{3-\mu}(T)}{\Gamma(4-\mu)}+\sum_{i=1}^{m}|\alpha_{i}|\psi_{a}(\eta_{i})
+\sum_{i=1}^{m}|\beta_{i}|\psi^{\prime}(\eta_{i})\Bigg)\\
&&+\big[\psi(t_{2})-\psi(t_{1})\big] \Bigg],\\
\end{eqnarray*}
which is independent of $u$ and tends to zero as $t_{2} - t_{1} \rightarrow 0$. Thus, $\mathcal{A}_{1}$ is equicontinuous. Hence, according to the
Arzela-Ascoli theorem, $\mathcal{A}_{1}(\mathcal{B}_{r_{0}})$ is a relatively compact set.

\textbf{Step 3.} $\mathcal{A}_{2}$ is continuous and $k$-contractive.\\

Let $u_{n}$ be a sequence converging to $u$. Then, by the assumption $(\rm{\mathcal{H}_{3}})$, for each $t \in J$, we have
\begin{eqnarray*}
\|\mathcal{A}_{2}u_{n}-\mathcal{A}_{2}u\|&\leq& N\frac{\psi_{a}^{\mu-1}(T)}{|\Delta|\Gamma(\mu)} \|u_{n}-u\|.
\end{eqnarray*}
In view of $(\rm{\mathcal{H}_{6}})$, this implies that $\mathcal{A}_{2}$ is continuous.\\
Next, since $k=\frac{\psi_{a}^{\mu-1}(T)}{|\Delta|\Gamma(\mu)}N\leq\Xi<1$, then $\mathcal{A}_{2}$ is $k$-contractive.\\
 \textbf{Step 4.} $\mathcal{A}$ is condensing.\\
Since $\mathcal{A}_{2}$ is a continuous, $k$-contraction and $\mathcal{A}_{1}$ is compact, by Lemma \ref{lem1.3}, $\mathcal{A}:\mathcal{B}_{r_{0}}\rightarrow \mathcal{B}_{r_{0}}$
with $\mathcal{A}=\mathcal{A}_{1}+\mathcal{A}_{2}$ is a condensing map on $\mathcal{B}_{r_{0}}$. Thus, we conclude by Theorem\ref{thm1.2} that $\mathcal{A}$ has a fixed point, which is a solution of $\psi$-Hilfer fractional boundary value problem \eqref{eq09}-\eqref{eq010}.
\end{proof}

\begin{theorem}\label{thmm3}
Assume that $(\rm{\mathcal{H}_{1}})$, $(\rm{\mathcal{H}_{3}})$, $(\rm{\mathcal{H}_{4}})$, and $(\rm{\mathcal{H}_{6}})$  hold. In addition, we assume that
\begin{itemize}
\item[($\mathcal{H}_{8}$)] There exist a function $p \in C(J, \mathbb{R_{+}})$ such that
    \[\big|f(t,x_{1},x_{2},x_{3})\big|\leq p(t),\ \text{for all}\ (t,x_{1},x_{2},x_{3})\in J\times \mathbb{R}\times\mathbb{R}\times\mathbb{R}.\]
      \end{itemize}
Then the $\psi$-Hilfer fractional boundary value problem \eqref{eq09}-\eqref{eq010} has at least one solution on $J$.
\end{theorem}
\begin{proof}
We first consider the operator $\mathcal{A}:\mathcal{C}(J,\mathbb{R})\rightarrow \mathcal{C}(J,\mathbb{R})$ with $\mathcal{A}=\mathcal{A}_{1}+\mathcal{A}_{2}$, where $\mathcal{A}_{1}$, $\mathcal{A}_{2}$ are defined by \eqref{eqn8} and \eqref{eqn9} respectively. Observe that Problem \eqref{eq09}-\eqref{eq010} has solutions if the operator $\mathcal{A}$ has fixed points.\\
Next, we will show that the operators $\mathcal{A}_{1}$ and $\mathcal{A}_{2}$ satisfy all the conditions of Theorem \ref{thmm3}. The proof will be split into numerous steps.\\
\textbf{Step 1.} The operator $\mathcal{A}_{1}$ is continuous.\\
Let $u_{n}$ be a sequence such that $u_{n}\rightarrow u$ in $\mathcal{C}(J,\mathbb{R})$. Then for each $t\in J$, we have
\begin{eqnarray*}
 \big|(\mathcal{A}_{1}u_{n})(t)-(\mathcal{A}_{1}u)(t)| &\leq& I_{a^{+}}^{\nu,\psi}|F_{u_{n}}(t)-F_{u}(t)|+\frac{\psi_{a}^{\mu-1}(t)}{|\Delta|\Gamma(\mu)}\Bigg [\sum_{i=1}^{m}|\alpha_{i}|I^{\nu,\psi}_{a^{+}}|F_{u_{n}}(\eta_{i})-F_{u}(\eta_{i})|\\
  &&+\sum_{i=1}^{m}|\beta_{i}|\psi(\eta_{i})
I^{\nu-1,\psi}_{a^{+}}|F_{u_{n}}(\eta_{i})-F_{u}(\eta_{i})|+I^{2-\mu+\nu,\psi}_{a^{+}}|F_{u_{n}}(T)-F_{u}(T)| \Bigg]\nonumber\\
&&+|\lambda| \Bigg[\frac{\psi_{a}^{\mu-1}(t)}{\Delta\Gamma(\mu)}\Bigg(I^{3-\mu,\psi}_{a^{+}}|u_{n}(T)-u(T)|+
\sum_{i=1}^{m}|\alpha_{i}|I^{1,\psi}_{a^{+}}|u_{n}(\eta_{i})-u(\eta_{i})|\\
&&+\sum_{i=1}^{m}|\beta_{i}|\psi^{\prime}(\eta_{i})
|u_{n}(\eta_{i})-u(\eta_{i})|\bigg)+ I^{1,\psi}_{a^{+}}|u_{n}(t)-u(t)|\Bigg]\\
&\leq& \|F_{u_{n}}-F_{u}\|\Bigg(\frac{\psi_{a}^{\nu}(T)}{\Gamma(\nu+1)}+\frac{\psi_{a}^{\mu-1}(T)}{|\Delta|\Gamma(\mu)
}\Bigg [\sum_{i=1}^{m}|\alpha_{i}|\frac{\psi_{a}^{\nu}(\eta_{i})}{\Gamma(\nu+1)}
+\sum_{i=1}^{m}|\beta_{i}|\psi(\eta_{i})\frac{\psi_{a}^{\nu-1}(\eta_{i})}{\Gamma(\nu)}\\
&&+\frac{\psi_{a}^{2-\mu+\nu}(T)}{\Gamma(3-\mu+\nu)} \Bigg]\Bigg)
+|\lambda| \Bigg[\frac{\psi_{a}^{\mu-1}(T)}{|\Delta|\Gamma(\mu)}\Bigg(
\frac{\psi_{a}^{3-\mu}(T)}{\Gamma(4-\mu)}+\sum_{i=1}^{m}|\alpha_{i}|\psi_{a}(\eta_{i})+
\sum_{i=1}^{m}|\beta_{i}|\psi^{\prime}(\eta_{i})\Bigg)\\
&&+\psi_{a}(T) \Bigg]\|u_{n}-u\|.\\
\end{eqnarray*}
Since $f$ is continuous, this implies that the operator $F_{u}$ is also continuous. Hence, we obtain
$$\|F_{u_{n}}-F_{u}\|\rightarrow 0\ \text{as}\ n\rightarrow \infty.$$
Thus, $\mathcal{A}_{1}$ is continuous.

\textbf{Step 2.} The operator $\mathcal{A}_{1}$ maps bounded sets into bounded sets in $\mathcal{C}(J,\mathbb{R})$.\\
To see this, let $u\in \mathcal{B}_{r}=\big\{u\in\mathcal{C}(J,\mathbb{R}):\|u\|\leq r\big\}$ ($r>0$), and $\|p\|=\sup_{t\in J}|p(t)|$. Then we have, for all $t\in J$

\begin{eqnarray*}
 \big|(\mathcal{A}_{1}u)(t)\big|&\leq& I_{a^{+}}^{\nu,\psi}|F_{u}(t)|+\frac{\psi_{a}^{\mu-1}(t)}{|\Delta|\Gamma(\mu)}\Bigg [\sum_{i=1}^{m}|\alpha_{i}|I^{\nu,\psi}_{a^{+}}|F_{u}(\eta_{i})|+\sum_{i=1}^{m}|\beta_{i}|\psi(\eta_{i})
I^{\nu-1,\psi}_{a^{+}}|F_{u}(\eta_{i})|+I^{2-\mu+\nu,\psi}_{a^{+}}|F_{u}(T)| \Bigg]\nonumber\\
&&+|\lambda| \Bigg[\frac{\psi_{a}^{\mu-1}(t)}{|\Delta|\Gamma(\mu)}\Bigg(I^{3-\mu,\psi}_{a^{+}}|u(T)|+
\sum_{i=1}^{m}|\alpha_{i}|I^{1,\psi}_{a^{+}}|u(\eta_{i})|+\sum_{i=1}^{m}|\beta_{i}|\psi^{\prime}(\eta_{i})
|u(\eta_{i})|\Bigg)+ I^{1,\psi}_{a^{+}}|u(t)|\Bigg]\\
&\leq& \|p\|\Bigg\{\frac{\psi_{a}^{\nu}(T)}{\Gamma(\nu+1)}+\frac{\psi_{a}^{\mu-1}(T)}{|\Delta|\Gamma(\mu)
}\Bigg [\sum_{i=1}^{m}|\alpha_{i}|\frac{\psi_{a}^{\nu}(\eta_{i})}{\Gamma(\nu+1)}\\
&&+\sum_{i=1}^{m}|\beta_{i}|\psi(\eta_{i})\frac{\psi_{a}^{\nu-1}(\eta_{i})}{\Gamma(\nu)}
+\frac{\psi_{a}^{2-\mu+\nu}(T)}{\Gamma(3-\mu+\nu)} \Bigg]\Bigg\}
+|\lambda|  r\Bigg[\frac{\psi_{a}^{\mu-1}(T)}{|\Delta|\Gamma(\mu)}\Bigg(
\frac{\psi_{a}^{3-\mu}(T)}{\Gamma(4-\mu)}+\sum_{i=1}^{m}|\alpha_{i}|\psi_{a}(\eta_{i})\\
&&+\sum_{i=1}^{m}|\beta_{i}|\psi^{\prime}(\eta_{i})\Bigg)
+\psi_{a}(T) \Bigg]:=L,
  \end{eqnarray*}
and consequently
$$\|\mathcal{A}_{1}u\|\leq L.$$

\textbf{Step 3.} The operator $\mathcal{A}_{1}$ maps bounded sets into equicontinuous sets in $\mathcal{C}(J,\mathbb{R})$.\\
Let $t_{1}, t_{2}\in J$ with $t_{1}< t_{2}$ and $u\in \mathcal{B}_{r}$. Then we have

\begin{eqnarray*}
\big|(\mathcal{A}_{1}u)(t_{2})|-|(\mathcal{A}_{1}u)(t_{1})\big|
  &\leq&  \Big|I_{a^{+}}^{\nu,\psi}F_{u}(t_{2})-I_{a^{+}}^{\nu,\psi}F_{u}(t_{1})\Big|+\frac{\psi_{a}^{\mu-1}(t_{2})-\psi_{a}^{\mu-1}(t_{1})}{|\Delta|\Gamma(\mu)}\Bigg [\sum_{i=1}^{m}|\alpha_{i}|I^{\nu,\psi}_{a^{+}}|F_{u}(\eta_{i})|\\
  &&+\sum_{i=1}^{m}|\beta_{i}|\psi(\eta_{i})
I^{\nu-1,\psi}_{a^{+}}|F_{u}(\eta_{i})|+I^{2-\mu+\nu,\psi}_{a^{+}}|F_{u}(T)| \Bigg]\nonumber\\
&&+|\lambda| \Bigg[\frac{\psi_{a}^{\mu-1}(t_{2})-\psi_{a}^{\mu-1}(t_{1})}{|\Delta|\Gamma(\mu)}\Bigg(I^{3-\mu,\psi}_{a^{+}}|u(T)|+
\sum_{i=1}^{m}|\alpha_{i}|I^{1,\psi}_{a^{+}}|u(\eta_{i})|\\
&&+\sum_{i=1}^{m}|\beta_{i}|\psi^{\prime}(\eta_{i})
|u(\eta_{i})|\Bigg)+ \big|I^{1,\psi}_{a^{+}}|u(t_{2})-I^{1,\psi}_{a^{+}}|u(t_{1})|\Bigg]\\
&\leq& \|p\|\Bigg\{\frac{\psi_{a}^{\nu}(t_{2})-\psi_{a}^{\nu}(t_{1})}{\Gamma(\nu+1)}+\frac{\psi_{a}^{\mu-1}(t_{2})-\psi_{a}^{\mu-1}(t_{1})}{|\Delta|\Gamma(\mu)
}\\
&&\times\Bigg [\sum_{i=1}^{m}|\alpha_{i}|\frac{\psi_{a}^{\nu}(\eta_{i})}{\Gamma(\nu+1)}
+\sum_{i=1}^{m}|\beta_{i}|\psi(\eta_{i})\frac{\psi_{a}^{\nu-1}(\eta_{i})}{\Gamma(\nu)}
+\frac{\psi_{a}^{2-\mu+\nu}(T)}{\Gamma(3-\mu+\nu)} \Bigg]\Bigg\}\\
&&+|\lambda|  r\Bigg[\frac{\psi_{a}^{\mu-1}(t_{2})-\psi_{a}^{\mu-1}(t_{1})}{|\Delta|\Gamma(\mu)}\Bigg(
\frac{\psi_{a}^{3-\mu}(T)}{\Gamma(4-\mu)}+\sum_{i=1}^{m}|\alpha_{i}|\psi_{a}(\eta_{i})
+\sum_{i=1}^{m}|\beta_{i}|\psi^{\prime}(\eta_{i})\Bigg)\\
&&+\big[\psi(t_{2})-\psi(t_{1})\big] \Bigg],\\
\end{eqnarray*}
The right hand side of the inequality above is independent of $u$ and tends to $0$ as $t_{2}\rightarrow t_{1}$.
Therefore, the operator $\mathcal{A}_{1}$ is equicontinuous.\\
\textbf{Step 4.} The operator $\mathcal{A}_{2}$ is a contraction.\\
In view of Step 3 of Theorem \ref{thmm2},  $\mathcal{A}_{2}$ is $k$-contractive, since
$$k=\frac{\psi_{a}^{\mu-1}(T)}{|\Delta|\Gamma(\mu)}N<1.$$
\textbf{Step 5.} A priori bounds.\\
Now it remains to show that the set
$$ \Upsilon=\Big\{u\in \mathcal{C}(J,\mathbb{R}): \gamma A_{2}\Big(\frac{u}{\gamma}\Big)+\gamma A_{1}(u)=u \Big\} $$ is bounded for some $\gamma \in (0,1).$

Let $u\in \Upsilon$, then $u=\gamma A_{2}\big(\frac{u}{\gamma}\big)+\gamma A_{1}(u)$ for some $0 < \gamma <1$. Thus, for each $t \in J$, we have
\begin{eqnarray*}
 |u(t)|&\leq& \gamma I_{a^{+}}^{\nu,\psi}|F_{u}(t)|+\gamma\frac{\psi_{a}^{\mu-1}(t)}{|\Delta|\Gamma(\mu)}\Bigg [\sum_{i=1}^{m}|\alpha_{i}|I^{\nu,\psi}_{a^{+}}|F_{u}(\eta_{i})|+\sum_{i=1}^{m}|\beta_{i}|\psi(\eta_{i})
I^{\nu-1,\psi}_{a^{+}}|F_{u}(\eta_{i})|+I^{2-\mu+\nu,\psi}_{a^{+}}|F_{u}(T)| \Bigg]\nonumber\\
&&+\gamma|\lambda| \Bigg[\frac{\psi_{a}^{\mu-1}(t)}{|\Delta|\Gamma(\mu)}\Bigg(I^{3-\mu,\psi}_{a^{+}}|u(T)|+
\sum_{i=1}^{m}|\alpha_{i}|I^{1,\psi}_{a^{+}}|u(\eta_{i})|+\sum_{i=1}^{m}|\beta_{i}|\psi^{\prime}(\eta_{i})
|u(\eta_{i})|\Bigg)+ I^{1,\psi}_{a^{+}}|u(t)|\Bigg]\\
&&+\gamma\frac{\psi_{a}^{\mu-1}(t)}{|\Delta|\Gamma(\mu)}\Big|g\Big(\frac{u(\xi)}{\gamma}\Big)\Big|\\
&\leq& I_{a^{+}}^{\nu,\psi}|F_{u}(t)|+\frac{\psi_{a}^{\mu-1}(t)}{|\Delta|\Gamma(\mu)}\Bigg [\sum_{i=1}^{m}|\alpha_{i}|I^{\nu,\psi}_{a^{+}}|F_{u}(\eta_{i})|+\sum_{i=1}^{m}|\beta_{i}|\psi(\eta_{i})
I^{\nu-1,\psi}_{a^{+}}|F_{u}(\eta_{i})|+I^{2-\mu+\nu,\psi}_{a^{+}}|F_{u}(T)| \Bigg]\nonumber\\
&&+|\lambda| \Bigg[\frac{\psi_{a}^{\mu-1}(t)}{|\Delta|\Gamma(\mu)}\Bigg(I^{3-\mu,\psi}_{a^{+}}|u(T)|+
\sum_{i=1}^{m}|\alpha_{i}|I^{1,\psi}_{a^{+}}|u(\eta_{i})|+\sum_{i=1}^{m}|\beta_{i}|\psi^{\prime}(\eta_{i})
|u(\eta_{i})|\Bigg)+ I^{1,\psi}_{a^{+}}|u(t)|\Bigg]\\
&&+\gamma\frac{\psi_{a}^{\mu-1}(t)}{|\Delta|\Gamma(\mu)}\Bigg[\Big|g\Big(\frac{u(\xi)}{\gamma}\Big)-g(a)\Big|+|g(a)|\Bigg]\\
&\leq& \|p\|\Bigg\{\frac{\psi_{a}^{\nu}(T)}{\Gamma(\nu+1)}+\frac{\psi_{a}^{\mu-1}(T)}{|\Delta|\Gamma(\mu)
}\Bigg [\sum_{i=1}^{m}|\alpha_{i}|\frac{\psi_{a}^{\nu}(\eta_{i})}{\Gamma(\nu+1)}+\sum_{i=1}^{m}|\beta_{i}|\psi(\eta_{i})\frac{\psi_{a}^{\nu-1}(\eta_{i})}{\Gamma(\nu)}
+\frac{\psi_{a}^{2-\mu+\nu}(T)}{\Gamma(3-\mu+\nu)} \Bigg]\Bigg\}\\
&&+|\lambda| \|u\|\Bigg[\frac{\psi_{a}^{\mu-1}(T)}{|\Delta|\Gamma(\mu)}\Bigg(
\frac{\psi_{a}^{3-\mu}(T)}{\Gamma(4-\mu)}+\sum_{i=1}^{m}|\alpha_{i}|\psi_{a}(\eta_{i})+\sum_{i=1}^{m}|\beta_{i}|\psi^{\prime}(\eta_{i})\Bigg)
+\psi_{a}(T) \Bigg]\\
&&+N\frac{\psi_{a}^{\mu-1}(T)}{|\Delta|\Gamma(\mu)}\|u\|+aN\frac{\psi_{a}^{\mu-1}(T)}{|\Delta|\Gamma(\mu)}.
\end{eqnarray*}
From the above we get
$$\|u\|\leq M:=\frac{\Theta\|p\|+aN\frac{\psi_{a}^{\mu-1}(T)}{|\Delta|\Gamma(\mu)}}{1-\Xi},$$
which shows that the set $\Upsilon$ is bounded, since $\Xi<1$. Thus the operators $\mathcal{A}_{1}$ and $\mathcal{A}_{2}$ satisfy all the conditions of Theorem \ref{thm1.3}. Hence, the operator $\mathcal{A}$ has a fixed point in $J$, which is the solution of the problem \eqref{eq09}-\eqref{eq010}.
\end{proof}

 \begin{theorem}\label{thmm4}
Let $f:J\times\mathbb{R}^{3}\rightarrow \mathbb{R}$ be a continuous function and hypotheses $(\rm{\mathcal{H}_{2}})$-$(\rm{\mathcal{H}_{4}})$ be satisfied. Then the $\psi$-Hilfer fractional boundary value problem \eqref{eq09}-\eqref{eq010} is Ulam-Hyers stable and consequently generalized Ulam-Hyres stable provided that
\[ \Omega+\frac{\psi_{a}^{\mu-1}(T)}{|\Delta|\Gamma(\mu)}N<1.\]
\end{theorem}
\begin{proof}
Let $\epsilon>0$ and $u \in C(J, \mathbb{R})$ be the solution of inequality \eqref{eq.UH}, and let $x\in C(J, \mathbb{R})$ be the unique solution of the following problem
\begin{equation*}
       \begin{cases} \big( ^{H}D^{\nu,\beta;\psi}_{a^{+}}+\lambda ^{H}D^{\nu-1,\beta,\psi}_{a^{+}}\big)x(t)= f(t,x(t),(\mathcal{V}x)(t), I_{a^{+}}^{2-\mu,\psi}x(t)),\ t\in J, \\
    x(a)=0,\  I^{2-\mu,\psi}_{a^{+}}x(T) =\sum_{i=1}^{m}\alpha_{i}x(\eta_{i})+\sum_{i=1}^{m}\beta_{i}x^{\prime}(\eta_{i})+g(x(\xi)).
       \end{cases}
       \end{equation*}
Since $u \in C(J, \mathbb{R})$ be the solution of inequality \eqref{eq.UH}, we have by (ii) of Remark \ref{rem 0.2.1}
     \[\big(^{H}D^{\nu,\beta;\psi}_{a^{+}}+\lambda ^{H}D^{\nu-1,\beta,\psi}_{a^{+}}\big)u(t)= f(t,u(t),(\mathcal{V}u)(t), I_{a^{+}}^{2-\mu,\psi}u(t))+z(t),\ t\in J,\]
     and
\begin{eqnarray}
u(t)&=&\int_{a}^{t}\mathcal{K}_{t}^{\nu}(s)F_{u}(s)ds+\frac{\psi_{a}^{\mu-1}(t)}{\Delta\Gamma(\mu)}
\Bigg[\sum_{i=1}^{m}\alpha_{i}\bigg(
\int_{a}^{\eta_{i}}\mathcal{K}_{\eta_{i}}^{\nu}(s)F_{u}(s)ds\bigg)\nonumber\\
&&+\sum_{i=1}^{m}\beta_{i}\psi(\eta_{i})\bigg(
\int_{a}^{\eta_{i}}\mathcal{K}_{\eta_{i}}^{\nu-1}(s)F_{u}(s)ds\bigg)-\int_{a}^{T}\mathcal{K}_{T}^{2-\mu+\nu}(s)
F_{u}(s)ds+g(u(\xi))\Bigg]\nonumber\\
&&+\lambda \Bigg[\frac{\psi_{a}^{\mu-1}(t)}{\Delta\Gamma(\mu)}\Bigg(\int_{a}^{T}\mathcal{K}_{T}^{3-\nu}(s)u(s)ds      -\bigg(\sum_{i=1}^{m}\alpha_{i}\int_{a}^{\eta_{i}}\psi^{\prime}(s)u(s)ds \nonumber\\
&&+\sum_{i=1}^{m}\beta_{i}\psi^{\prime}(\eta_{i})
u(\eta_{i})\bigg)\Bigg)-\int_{a}^{t}\psi^{\prime}(s)u(s)ds \Bigg]\nonumber\\
&&+\int_{a}^{t}\mathcal{K}_{t}^{\nu}(s)z(s)ds+\frac{\psi_{a}^{\mu-1}(t)}{\Delta\Gamma(\mu)}
\Bigg[\sum_{i=1}^{m}\alpha_{i}\bigg(
\int_{a}^{\eta_{i}}\mathcal{K}_{\eta_{i}}^{\nu}(s)z(s)ds\bigg)\nonumber\\
&&+\sum_{i=1}^{m}\beta_{i}\psi(\eta_{i})\bigg(
\int_{a}^{\eta_{i}}\mathcal{K}_{\eta_{i}}^{\nu-1}(s)z(s)ds\bigg)-\int_{a}^{T}\mathcal{K}_{T}^{2-\mu+\nu}(s)
z(s)ds\Bigg]\label{eq.UH1}.
\end{eqnarray}
From Lemma \ref{lem1}, we obtain $x(t) =\int_{a}^{t}\mathcal{K}_{t}^{\nu}(s)F_{x}(s)ds+\chi_{x}(t)$, where
\begin{eqnarray}
\chi_{x}(t)&=&\frac{\psi_{a}^{\mu-1}(t)}{\Delta\Gamma(\mu)}
\Bigg[\sum_{i=1}^{m}\alpha_{i}\bigg(
\int_{a}^{\eta_{i}}\mathcal{K}_{\eta_{i}}^{\nu}(s)F_{x}(s)ds\bigg)\nonumber\\
&&+\sum_{i=1}^{m}\beta_{i}\psi(\eta_{i})\bigg(
\int_{a}^{\eta_{i}}\mathcal{K}_{\eta_{i}}^{\nu-1}(s)F_{x}(s)ds\bigg)-\int_{a}^{T}\mathcal{K}_{T}^{2-\mu+\nu}(s)
F_{x}(s)ds+g(x(\xi))\Bigg]\nonumber\\
&&+\lambda \Bigg[\frac{\psi_{a}^{\mu-1}(t)}{\Delta\Gamma(\mu)}\Bigg(\int_{a}^{T}\mathcal{K}_{T}^{3-\nu}(s)x(s)ds      -\bigg(\sum_{i=1}^{m}\alpha_{i}\int_{a}^{\eta_{i}}\psi^{\prime}(s)x(s)ds \nonumber\\
&&+\sum_{i=1}^{m}\beta_{i}\psi^{\prime}(\eta_{i})
x(\eta_{i})\bigg)\Bigg)-\int_{a}^{t}\psi^{\prime}(s)x(s)ds \Bigg]\nonumber
\end{eqnarray}
By using (i) of Remark \ref{rem 0.2.1} and \eqref{eq.UH1}, we obtain the following estimation
\begin{eqnarray}
\Big|u(t)-\chi_{u}(t)- \int_{a}^{t}\mathcal{K}_{t}^{\nu}(s)F_{u}(s)ds\Big|&=&\Bigg| \int_{a}^{t}\mathcal{K}_{t}^{\nu}(s)z(s)ds+ \frac{\psi_{a}^{\mu-1}(t)}{\Delta\Gamma(\mu)}
\Bigg[\sum_{i=1}^{m}\alpha_{i}\bigg(
\int_{a}^{\eta_{i}}\mathcal{K}_{\eta_{i}}^{\nu}(s)z(s)ds\bigg)\nonumber\\
&&+\sum_{i=1}^{m}\beta_{i}\psi(\eta_{i})\bigg(
\int_{a}^{\eta_{i}}\mathcal{K}_{\eta_{i}}^{\nu-1}(s)z(s)ds\bigg)-\int_{a}^{T}\mathcal{K}_{T}^{2-\mu+\nu}(s)
z(s)ds\Bigg] \Bigg|\nonumber\\
&\leq&\epsilon \Bigg (\frac{\psi_{a}^{\nu}(T)}{\Gamma(\nu+1)}+\frac{\psi_{a}^{\mu-1}(T)}{|\Delta|\Gamma(\mu)
}\Bigg [\sum_{i=1}^{m}|\alpha_{i}|\frac{\psi_{a}^{\nu}(\eta_{i})}{\Gamma(\nu+1)}
+\sum_{i=1}^{m}|\beta_{i}|\psi(\eta_{i})\frac{\psi_{a}^{\nu-1}(\eta_{i})}{\Gamma(\nu)}\nonumber\\
&&+\frac{\psi_{a}^{2-\mu+\nu}(T)}{\Gamma(3-\mu+\nu)} \Bigg] \Bigg)\nonumber\\
&\leq& \epsilon \Theta \label{eq.UH2}.
\end{eqnarray}
From \eqref{eq.UH2}, for $t\in J$, we have
\begin{eqnarray*}
|u(t)-x(t)|&=&\Big|u(t)-\chi_{x}(t)- \int_{a}^{t}\mathcal{K}_{t}^{\nu}(s)F_{x}(s)ds\Big|\\
&\leq&\Big|u(t)-\chi_{u}(t)- \int_{a}^{t}\mathcal{K}_{t}^{\nu}(s)F_{u}(s)ds\Big|+\int_{a}^{t}\mathcal{K}_{t}^{\nu}(s)|F_{u}(s)-F_{x}(s)|ds+\big|\chi_{u}(t)-\chi_{x}(t) \big|\\
&\leq& \epsilon \Theta+ \frac{\psi_{a}^{\nu}(T)}{\Gamma(\nu+1)}\|u-x\|\Lambda(T,\mu)+\Lambda(T,\mu)\frac{\psi_{a}^{\mu-1}(T)}{|\Delta|\Gamma(\mu)}\|u-x\|\\
&&\times \Bigg [\sum_{i=1}^{m}|\alpha_{i}|\frac{\psi_{a}^{\nu}(\eta_{i})}{\Gamma(\nu+1)}+\sum_{i=1}^{m}|\beta_{i}|\psi(\eta_{i})\frac{\psi_{a}^{\nu-1}(\eta_{i})}{\Gamma(\nu)}
+\frac{\psi_{a}^{2-\mu+\nu}(T)}{\Gamma(3-\mu+\nu)} \Bigg]+N\|u-x\|\frac{\psi_{a}^{\mu-1}(T)}{|\Delta|\Gamma(\mu)}\\
&&+|\lambda|\|u-x\|\Bigg[\frac{\psi_{a}^{\mu-1}(T)}{|\Delta|\Gamma(\mu)}\Bigg(
\frac{\psi_{a}^{3-\mu}(T)}{\Gamma(4-\mu)}+\sum_{i=1}^{m}|\alpha_{i}|\psi_{a}(\eta_{i})+\sum_{i=1}^{m}|\beta_{i}|\psi^{\prime}(\eta_{i})\Bigg)
+\psi_{a}(T) \Bigg]\\
&\leq& \epsilon \Theta+ \frac{\psi_{a}^{\nu}(T)}{\Gamma(\nu+1)}\|u-x\|\Lambda(T,\mu)+\Lambda(T,\mu)\|u-x\|
\bigg(\Theta-\frac{\psi_{a}^{\nu}(T)}{\Gamma(\nu+1)}\bigg)\\
&&+ N\|u-x\|\frac{\psi_{a}^{\mu-1}(T)}{|\Delta|\Gamma(\mu)}+|\lambda|\|u-x\|\Phi \\
&\leq&\epsilon \Theta+\bigg(\Omega+\frac{\psi_{a}^{\mu-1}(T)}{|\Delta|\Gamma(\mu)}N \bigg)\|u-x\|.
\end{eqnarray*}
In consequence, it follows that
\[\|u-x\|\leq \frac{\epsilon \Theta}{1-\Big(\Omega+\frac{\psi_{a}^{\mu-1}(T)}{|\Delta|\Gamma(\mu)}N \Big)}.\]
If we set $c_{f}=\frac{\Theta}{1-\Big(\Omega+\frac{\psi_{a}^{\mu-1}(T)}{|\Delta|\Gamma(\mu)}N \Big)}$, then the $\psi$-Hilfer fractional boundary value problem \eqref{eq09}-\eqref{eq010} is Ulam-Hyers stable. Furthermore, by choosing $\varphi_{f}(\epsilon)=\frac{\epsilon \Theta}{1-\Big(\Omega+\frac{\psi_{a}^{\mu-1}(T)}{|\Delta|\Gamma(\mu)}N \Big)}$ with $\varphi_{f}(0)=0$, the problem \eqref{eq09}-\eqref{eq010} is also generalized Ulam-Hyers stable.
\end{proof}

\section{Illustrative examples}
\begin{exmp}
Consider the following $\psi$-Hilfer fractional integro-differential equation with nonlocal boundary conditions
\begin{equation}\label{eq4.1}
       \begin{cases}\Big( ^{H}D^{\frac{3}{2},\frac{1}{2};1-e^{-t\sqrt{2}}}_{0^{+}}+ \frac{1}{100} ^{H}D^{\frac{1}{2},\frac{1}{2},1-e^{-t\sqrt{2}}}_{0^{+}}\Big)u(t)= f\Big(t,u(t),(\mathcal{V}u)(t), I_{0^{+}}^{\frac{1}{4},1-e^{-t\sqrt{2}}}u(t)\Big) ,\  t \in J,\\
    u(0)=0,\  I^{\frac{1}{4},1-e^{-t\sqrt{2}}}_{0^{+}}u\big(\frac{7}{6}\big) =u\big(\frac{1}{6}\big)+2u\big(\frac{5}{6}\big)+\frac{1}{3}u^{\prime}\big(\frac{1}{6}\big)
     +\frac{2}{5}u^{\prime}\big(\frac{5}{6}\big)+g(u(1)),
       \end{cases}
       \end{equation}
where $\nu=\frac{3}{2}$, $\beta=\frac{1}{2}$, $\psi(t)=1-e^{-t\sqrt{2}}$, $a=0$, $T=\frac{7}{6}$, $\lambda=\frac{1}{100}$, $m=2$, $\mu=\frac{7}{4}$, $\alpha_{1}=1$, $\alpha_{2}=2$, $\beta_{1}=\frac{1}{3}$, $\beta_{2}=\frac{2}{5}$, $\xi=1$, $\eta_{1}=\frac{1}{6}$, $\eta_{2}=\frac{5}{6}$, and $g(t)=\frac{1}{10}\ln(1+t^{2}\sqrt{3})$.\\

\item[(i)] Consider the nonlinear function
\begin{eqnarray*}
f\Big(t,u(t),(\mathcal{V}u)(t), I_{0^{+}}^{\frac{1}{4},\psi}u(t)\Big)&=&\frac{\cos t}{1+t}+\frac{e^{-\sin^{2}t}}{\sqrt{25+t}} \frac{|u(t)|}{2+|u(t)|}\\
&&+\frac{(2t-1)^{2}}{6}\Bigg(\cos\Big[\frac{\pi}{2}(\mathcal{V}u)(t)\Big]
+\frac{\big|I_{0^{+}}^{\frac{1}{4},\psi}u(t)\big|}{3+\big|I_{0^{+}}^{\frac{1}{4},\psi}u(t)\big|} \Bigg),
\end{eqnarray*}
with $(\mathcal{V}u)(t)=u(\tau(t))$, $\tau(t)=\frac{2t}{1+t}$.\\
For all $u_{1}, u_{2}, v_{1}, v_{2}, w_{1}, w_{2}$, and $t\in J$, we have
 \begin{eqnarray*}
 |f(t,u_{1},v_{1},w_{1})-f(t,u_{2},v_{2},w_{2})|&\leq& \frac{e^{-\sin^{2}t}}{\sqrt{25+t}} \Bigg |\frac{|u_{1}|}{2+|u_{1}|}-\frac{|u_{2}|}{2+|u_{2}|}\Bigg|+\frac{(2t-1)^{2}}{6}\\
 &&\times \Bigg[\bigg|\cos\Big(\frac{\pi}{2}v_{1}\Big)-\cos\Big(\frac{\pi}{2}v_{2}\Big)\bigg|
+\bigg|\frac{\big|w_{1}\big|}{3+\big|w_{1}\big|}-\frac{\big|w_{2}\big|}{3+\big|w_{2}\big|}\bigg| \Bigg]\\
&\leq& 2 \frac{e^{-\sin^{2}t}}{\sqrt{25+t}} \frac{\big||u_{1}|-|u_{2}|\big|}{\big(2+|u_{1}|\big)\big(2+|u_{2}|\big)}+\frac{(2t-1)^{2}}{6}\\
&&\times \Bigg[\frac{\pi}{2}|v_{1}-v_{2}|+\frac{3\big  ||w_{1}|-|w_{2}|\big|}{\big(3+|w_{1}|\big)\big(3+|w_{2}|\big)} \Bigg]\\
&\leq& \frac{1}{2}\frac{e^{-\sin^{2}t}}{\sqrt{25+t}} |u_{1}-u_{2}|+\frac{\pi(2t-1)^{2}}{12}|v_{1}-v_{2}|+\frac{(2t-1)^{2}}{18}|w_{1}-w_{2}|\\
&\leq& \frac{1}{10}|u_{1}-u_{2}|+\frac{4\pi}{27}|v_{1}-v_{2}|+\frac{8}{81}|w_{1}-w_{2}|,
  \end{eqnarray*}
  and
  \[|g(u_{1})-g(u_{2})|\leq \frac{\sqrt[4]{3}}{10}|u_{1}-u_{2}|.\]
So, the assumptions $(\rm{\mathcal{H}_{1}})$-$(\rm{\mathcal{H}_{4}})$ are satisfied with $l^{\star}_{1}=\frac{1}{10}$, $l^{\star}_{2}=\frac{4\pi}{27}$, $l^{\star}_{3}=\frac{8}{81}$, and $N=\frac{\sqrt[4]{3}}{10}$.
Using the given values, we have $\Lambda\Big(\frac{7}{6},\frac{7}{4}\Big)\thickapprox0.668728$,\ $\Delta\thickapprox -1.94081$,\ $\Theta\thickapprox1.34089$,\ $\Phi\thickapprox 1.73186$,\ $\Omega\thickapprox0.9140092$.\\
Hence
 \begin{equation*}
 \Omega+\frac{\psi_{a}^{\mu-1}(T)}{|\Delta|\Gamma(\mu)}N\thickapprox0.976884<1.
  \end{equation*}

Since all the assumptions of Theorem \ref{thmm1} are fulfilled, the $\psi$-Hilfer fractional boundary value problem \eqref{eq4.1} has a unique solution on $J$. Furthermore, we have
$$c_{f}=\frac{ \Theta}{1-\Big(\Omega+\frac{\psi_{a}^{\mu-1}(T)}{|\Delta|\Gamma(\mu)}N \Big)}\approx 58.007>0.$$
Hence, by Theorem \ref{thmm4}, the $\psi$-Hilfer fractional boundary value problem \eqref{eq09}-\eqref{eq010} is both Ulam–Hyers stable and also generalized Ulam–Hyers stable.\\
\item[(ii)] Consider the nonlinear function
 \begin{equation*}
f(t,x_{1},x_{2},x_{3})=\frac{e^{-\sqrt{t}}}{1+t}ax_{1}+\frac{b(t-t^{^2})}{2+t}\bigg(3+\frac{|x_{2}|}{1+|x_{2}|} \bigg)+\frac{ce^{\sin\big(\frac{3\pi t}{7} \big)}}{\sqrt{16+t^{2}}}\bigg(2x_{3}^{3}+\frac{1}{4} \bigg)
 \end{equation*}
For $x_{i}\in\mathbb{R}, i = 1,2,3$, and $t\in J$, we can estimate
 \begin{equation*}
\big| f(t,x_{1},x_{2},x_{3})\big|\leq ae^{-\sqrt{t}}|x_{1}|+\frac{b(t-t^{^2})}{2}\bigg(3+\frac{|x_{2}|}{1+|x_{2}|} \bigg)+\frac{ce^{\sin\big(\frac{3\pi t}{7} \big)}}{4}\bigg(2x_{3}^{3}+\frac{1}{4} \bigg).
\end{equation*}
This means that condition $(\rm{\mathcal{H}_{5}})$ is valid with $p_{1}(t)=ae^{-\sqrt{t}}$, $p_{2}(t)=\frac{b(t-t^{^2})}{2}$, $p_{3}(t)=\frac{ce^{\sin\big(\frac{3\pi t}{7} \big)}}{4}$, $p_{1}^{\star}=a$, $p_{2}^{\star}=\frac{b}{8}$, $p_{3}^{\star}=\frac{c}{4}e$, and $\varphi_{1}(t)=t$, $\varphi_{2}(t)=3+\frac{t}{1+t}$, $\varphi_{3}(t)=2t^{3}+\frac{1}{4}$.\\
Further, we have that $(\rm{\mathcal{H}_{6}})$ holds, since
\[\Xi:=|\lambda|\Phi +N\frac{\psi_{a}^{\mu-1}(T)}{|\Delta|\Gamma(\mu)}\thickapprox0.0801944<1.\]
Putting
\begin{eqnarray*}
\zeta(r)&=&\frac{\Big(\sum_{i=1}^{2}p_{i}^{\star}\varphi_{i}(r)+p_{3}^{\star}\varphi_{3}\Big(\frac{\psi_{a}^{2-\mu}(T)}{\Gamma(3-\mu)}r\Big) \Big)\Theta+\frac{\psi_{a}^{\mu-1}(T)}{|\Delta|\Gamma(\mu)}Na}{1-\Xi}-r\\
&=&\frac{\Big(ar+\frac{b}{8}\big(3+\frac{r}{1+r}\big)+\frac{c}{4}e\Big[\frac{1}{4}
+2\Big(\frac{\big(1-e^{-\frac{7}{6}\sqrt{2}}\big)^{\frac{1}{4}}r}{\Gamma(\frac{5}{4})}\Big)^{3}\Big] \Big)\Theta}{1-\Xi}-r.
\end{eqnarray*}
Selecting $a=\frac{37}{500}$, $b=\frac{29}{125}$, $c=\frac{3}{10}$. By using the graphical representation of the function $\zeta(r)$ depicted in Figure 1, we obtain
\[ \zeta(r)<0,\ \forall r\in \bigg[\frac{1}{4}, \frac{39}{40} \bigg].\]
Then, $(\rm{\mathcal{H}_{7}})$ is satisfied. Thus, all the conditions of Theorem \ref{thmm2} are
satisfied, and consequently the $\psi$-Hilfer fractional boundary value problem \eqref{eq4.1} has a solution on $J$.\\
\begin{figure}[h!]
\begin{center}
\includegraphics[totalheight=2.5in]{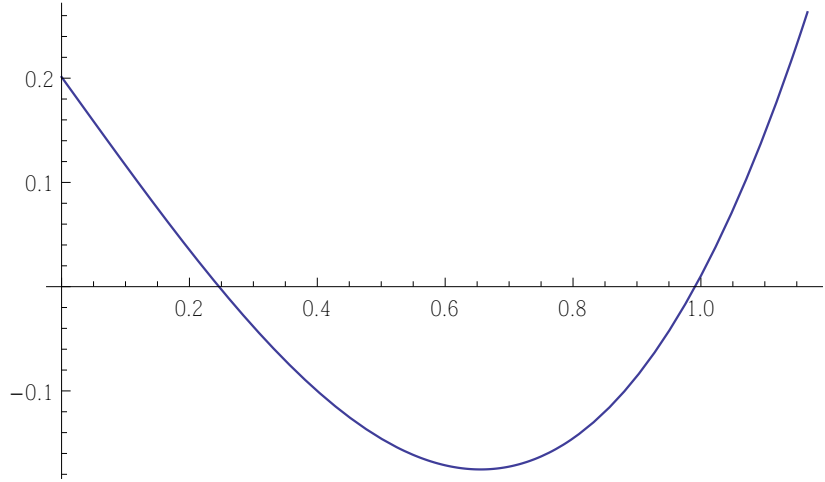}
 \caption{Graph of the function $\zeta(r)$,\ $ r\in \big[0,\frac{7}{6}\big].$}
\end{center}
\end{figure}

\item[(iii)] Consider the function $f\Big(t,u(t),(\mathcal{V}u)(t), I_{0^{+}}^{\frac{1}{4},\psi}u(t)\Big)=1$, $g(u(\xi))=1$, and $p(t)=e^{t}$. Then, all the assumptions of Theorem \ref{thmm3} are satisfied and by Lemma \ref{lem1} the solution of the $\psi$-Hilfer fractional boundary value problem \eqref{eq4.1} is given by
\begin{eqnarray*}\label{eq3}
u(t) &=&\frac{[\psi(t)]^{\nu}}{\Gamma(\nu+1)}+\frac{[\psi(t)]^{\mu-1}}{\Delta\Gamma(\mu)}
\Bigg[\sum_{i=1}^{m}\alpha_{i}
\frac{[\psi(\eta_{i})]^{\nu}}{\Gamma(\nu+1)}+\sum_{i=1}^{m}\beta_{i}
\frac{[\psi(\eta_{i})]^{\nu}}{\Gamma(\nu)}
-\frac{[\psi(T)]^{2-\mu+\nu}}{\Gamma(3-\mu+\nu)}+1\Bigg]\nonumber\\
&&+\lambda \Bigg[\frac{[\psi(t)]^{\mu-1}}{\Delta\Gamma(\mu)}\Bigg(\int_{0}^{T}\frac{\psi^{\prime}(s)(\psi(T)-\psi(s))^{2-\nu}}{\Gamma(3-\nu)}  u(s)ds-\sum_{i=1}^{m}\alpha_{i}\int_{0}^{\eta_{i}}\psi^{\prime}(s)u(s)ds \nonumber\\
&&-\sum_{i=1}^{m}\beta_{i}\psi^{\prime}(\eta_{i})
u(\eta_{i})\Bigg)-\int_{0}^{t}\psi^{\prime}(s)u(s)ds \Bigg],\ t\in[0,T].
\end{eqnarray*}
\item[(1)] Setting $\psi_{1}(t)=3^{(t^{\rho}+2t)}-1$ with $1<\rho\leq2$, we have that $\psi_{1}$ is a nondecreasing positive function with $\psi_{1}^{\prime}(t)\neq 0$ for all $t\in[0,T]$, and the solution of the $\psi$-Hilfer fractional boundary value problem \eqref{eq4.1} is given by
    \begin{eqnarray*}\label{eq4}
u(t)&=&\frac{[-1+3^{t^{\rho}+2t}]^{\frac{3}{2}}}{\Gamma(\frac{5}{2})}+\frac{[-1+3^{t^{\rho}+2t}]^{\frac{3}{4}}}{\Delta\Gamma\big(\frac{7}{4}\big)}
\Bigg[\frac{1}{\Gamma(\frac{5}{2})}\bigg(\Big[-1+3^{\big(\frac{1}{6}\big)^{\rho}+\frac{1}{3}}\Big]^{\frac{3}{2}}+2\Big[-1+3^{\big(\frac{5}{6}\big)^{\rho}+\frac{5}{3}}\Big]^{\frac{3}{2}} \bigg)\\
&&+\frac{1}{\Gamma(\frac{3}{2})} \bigg(\frac{1}{3}\Big[-1+3^{\big(\frac{1}{6}\big)^{\rho}+\frac{1}{3}}\Big]^{\frac{3}{2}}
+\frac{2}{5}\Big[-1+3^{\big(\frac{5}{6}\big)^{\rho}+\frac{5}{3}}\Big]^{\frac{3}{2}} \bigg)
-\frac{\Big[-1+3^{\big(\frac{7}{6}\big)^{\rho}+\frac{7}{3}}\Big]^{\frac{7}{4}}}{\Gamma(\frac{11}{4})}+1\Bigg]\nonumber\\
&&+\lambda \ln3 \Bigg[\frac{[-1+3^{t^{\rho}+2t}]^{\frac{3}{4}}}{\Delta\Gamma\big(\frac{7}{4}\big)}\Bigg(\Gamma\Big(\frac{3}{2}\Big)\int_{0}^{\frac{7}{6}}(\rho s^{\rho-1}+2) 3^{s^{\rho}+2s}\Big[3^{(\frac{7}{6})^{\rho}+\frac{7}{3}}-3^{s^{\rho}+2s}\Big]^{\frac{1}{2}}u(s)ds\\
&&-\int_{0}^{\frac{1}{6}}(\rho s^{\rho-1}+2) 3^{s^{\rho}+2s}u(s)ds -2\int_{0}^{\frac{5}{6}}(\rho s^{\rho-1}+2) 3^{s^{\rho}+2s}u(s)ds\nonumber\\
&&-\frac{1}{3}\bigg(\rho \Big(\frac{1}{6}\Big)^{\varrho-1}+2\bigg) 3^{(\frac{1}{6})^{\rho}+\frac{1}{3}}u\Big(\frac{1}{6}\Big)-\frac{2}{5}\bigg(\rho \Big(\frac{5}{6}\Big)^{\rho-1}+2\bigg) 3^{(\frac{5}{6})^{\rho}+\frac{5}{3}}u\Big(\frac{5}{6}\Big)\Bigg)\\
&&-\int_{0}^{t}(\rho s^{\rho-1}+2) 3^{s^{\rho}+2s}u(s)ds \Bigg],\ \ \ t\in\bigg[0,\frac{7}{6}\bigg].
\end{eqnarray*}
\item[(2)] If $\psi_{2}(t)=\tan\big(\frac{\pi t \sqrt{\rho}}{4}\big)$ with $1<\rho\leq2$, then the solution of the $\psi$-Hilfer fractional boundary value problem \eqref{eq4.1} is given by
 \begin{eqnarray*}\label{eq5}
 u(t)&=&\frac{\big[\tan\big(\frac{\pi t \sqrt{\rho}}{4}\big)\big]^{\frac{3}{2}}}{\Gamma(\frac{5}{2})}+\frac{\big[\tan\big(\frac{\pi t \sqrt{\rho}}{4}\big)\big]^{\frac{3}{4}}}{\Delta\Gamma\big(\frac{7}{4}\big)}
\Bigg[\frac{1}{\Gamma(\frac{5}{2})}\bigg(\Big[\tan\big(\frac{\pi\sqrt{\rho}}{24}\big)\Big]^{\frac{3}{2}}
+2\Big[\tan\big(\frac{5\pi \sqrt{\rho}}{24}\big)\Big]^{\frac{3}{2}} \bigg)\\
&&+\frac{1}{\Gamma(\frac{3}{2})} \bigg(\frac{1}{3}\Big[\tan\big(\frac{\pi\sqrt{\rho}}{24}\big)\Big]^{\frac{3}{2}}
+\frac{2}{5}\Big[\tan\big(\frac{5\pi \sqrt{\rho}}{24}\big) \Big]^{\frac{3}{2}} \bigg)
-\frac{\Big[\tan\big(\frac{7\pi \sqrt{\rho}}{24}\big)\Big]^{\frac{7}{4}}}{\Gamma(\frac{11}{4})}+1\Bigg]\nonumber\\
&&+\frac{\lambda \pi \sqrt{\rho}}{4} \Bigg[\frac{\big[\tan\big(\frac{\pi t \sqrt{\rho}}{4}\big)\big]^{\frac{3}{4}}}{\Delta\Gamma\big(\frac{7}{4}\big)} \Bigg(\Gamma\Big(\frac{3}{2}\Big)\int_{0}^{\frac{7}{6}}\Big[1+\tan^{2}\big(\frac{\pi s \sqrt{\rho}}{4}\big)\Big] \Big[ \tan\big(\frac{7\pi  \sqrt{\rho}}{24}\big)-\tan\big(\frac{\pi s \sqrt{\rho}}{4}\big)\Big]^{\frac{1}{2}}u(s)ds\\
&&-\int_{0}^{\frac{1}{6}} \Big[1+\tan^{2}\big(\frac{\pi s \sqrt{\rho}}{4}\big)\Big]u(s)ds -2\int_{0}^{\frac{5}{6}}\Big[1+\tan^{2}\big(\frac{\pi s \sqrt{\rho}}{4}\big)\Big]u(s)ds\nonumber\\
&&-\frac{1}{3}\Big[1+\tan^{2}\big(\frac{\pi  \sqrt{\rho}}{24}\big)\Big] u\Big(\frac{1}{6}\Big)-\frac{2}{5} \Big[1+\tan^{2}\big(\frac{5\pi  \sqrt{\rho}}{24}\big)\Big] u\Big(\frac{5}{6}\Big)\Bigg)\\
&&-\int_{0}^{t}  \Big[1+\tan^{2}\big(\frac{\pi s \sqrt{\rho}}{4}\big)\Big]u(s)ds \Bigg],\ \ \ t\in\bigg[0,\frac{7}{6}\bigg].
 \end{eqnarray*}
\end{exmp}
The graph of the solution of the $\psi$-Hilfer fractional boundary value problem \eqref{eq4.1} for different values of $\rho=\frac{11}{10},\frac{13}{10},\frac{15}{10},\frac{17}{10},\frac{19}{10},\frac{20}{10}$, and $\lambda=0, \lambda=10^{-2}$ involving a variety of functions $\psi_{1}(t)=3^{(t^{\rho}+2t)}-1, \psi_{2}(t)=\tan\big(\frac{\pi t \sqrt{\rho}}{4}\big)$ is depicted in Figures 2-....  .

\end{document}